\newif\iffigure
\figuretrue 

\documentclass[11pt]{article}

\usepackage{amsmath, amssymb, amsthm, verbatim,enumerate,bbm,color}
\usepackage{indentfirst, mathtools}
\usepackage{nicefrac,xfrac}
\usepackage{hyperref}
\usepackage[nameinlink]{cleveref}
\usepackage{tikz}
\usetikzlibrary{shapes.misc}
\usepackage{tkz-graph}

\title{Tur\'an-type problems for long cycles in random and pseudo-random graphs}

\author{Michael Krivelevich \thanks{School of Mathematical Sciences,
	Raymond and Beverly Sackler Faculty of Exact Sciences, Tel Aviv
	University, Tel Aviv, 6997801, Israel. Email: krivelev@post.tau.ac.il.
	Partially supported by USA-Israel BSF grants 2014361 and 2018267, and by ISF grant 1261/17.} 
	\and Gal Kronenberg
	\thanks{School of Mathematical Sciences, Raymond and Beverly Sackler Faculty of Exact Sciences, Tel Aviv University, Tel Aviv, 6997801, Israel, and Mathematical Institute, University of Oxford, Oxford, UK. Email: kronenberg@maths.ox.ac.uk.}
	\and Adva Mond \thanks{School of Mathematical Sciences, Raymond and Beverly Sackler Faculty of Exact Sciences, Tel Aviv University, Tel Aviv, 6997801, Israel, and Department of Pure Mathematics and Mathematical Statistics, Centre for Mathematical Sciences, University of Cambridge, Wilberforce Road, Cambridge CB3 0WB, UK. Email: advamond@gmail.com.}}

\date{\today}
\allowdisplaybreaks

\DeclarePairedDelimiter\ceil{\lceil}{\rceil}
\DeclarePairedDelimiter\floor{\lfloor}{\rfloor}

\theoremstyle{plain}
\newtheorem{theorem}{Theorem}[section]
\crefname{theorem}{Theorem}{Theorems}
\newtheorem{lemma}[theorem]{Lemma}
\newtheorem{claim}[theorem]{Claim}
\newtheorem{proposition}[theorem]{Proposition}

\newtheorem{corollary}[theorem]{Corollary}

\newtheorem{definition}[theorem]{Definition}

\newtheorem{fact}[theorem]{Fact}

\theoremstyle{remark}
\newtheorem{remark}[theorem]{Remark}

\newcommand{\whp}{w.h.p.\ }

\newcommand{\mbinom}[2]{\Bigl(\begin{array}{@{}c@{}}#1\\#2\end{array}\Bigr)}

\RequirePackage[normalem]{ulem}
\RequirePackage{color}\definecolor{RED}{rgb}{1,0,0}\definecolor{BLUE}{rgb}{0,0,1}

\begin{document}
\maketitle

\begin{abstract}
	We study the Tur\'an number of long cycles in random graphs and in pseudo-random graphs.
	Denote by $ex(G(n,p),H)$ the random variable counting the number of edges in a largest subgraph of
	$G(n,p)$ without a copy of $H$.
	We determine the asymptotic value of $ex(G(n,p), C_t)$ where $C_t$ is a cycle of length $t$, 
	for $p\geq \frac Cn$ and 
	$A \log n \leq t \leq (1 - \varepsilon)n$.
	The typical behavior of $ex(G(n,p), C_t)$ depends substantially on the parity of $t$.
	In particular, our results match the classical result of Woodall on the Tur\'an number of long cycles, and can be seen as its random version, showing that the transference principle holds here as well.
	In fact, our techniques apply in a more general sparse pseudo-random
	setting.
	We also prove a robustness-type result, showing the likely existence of cycles of prescribed lengths in a random subgraph of a graph with a nearly optimal density. Finally, we also present further applications of our main tool (the Key Lemma) for proving results on Ramsey-type problems about cycles in sparse random graphs. 
\end{abstract}

\section{Introduction}

One of the most central topics  in extremal graph theory is the so-called Tur\'an-type problems.
Recall that $ex(n,H)$ denotes the maximum possible number of edges in a graph on $n$ vertices without having $H$ as a subgraph.
Determining the value of $ex(n,H)$ for a fixed graph $H$ has become one of the most central problems in extremal combinatorics and there is a rich literature investigating it. 
Mantel~\cite{Mantel} proved in 1907 that $ex(n, K_3) = \floor{\tfrac{n^2}4}$; Tur\'an~\cite{Turan} found the value of $ex(n, K_t)$ for $t\geq 3$ in 1941.
In 1968, Simonovits~\cite{Simon} showed that the result of Mantel can be extended for an odd cycle of a fixed length, that is, $ex(n,C_{2t+1})=\floor{\frac {n^2}4}$, where the extremal example is the complete bipartite graph\footnote{Throughout the paper we denote by $P_t$ and $C_t$ the path and the cycle of length $t$ (i.e., the path and the cycle with $t$ edges), respectively.
}.
For the general case, it was proved in 1946 by Erd\H{o}s and Stone~\cite{ErdosStone} that $ex(n,H)=\left(1-\frac 1{\chi(H)-1}+o(1)\right)\mbinom n2$, where $\chi(H)$ is the chromatic number of the fixed graph $H$.
Note that when $H$ is a graph with chromatic number 2,  an even cycle for instance, then from the above result we can only obtain that $ex(n,H)=o(n^2)$.
Bondy and Simonovits~\cite{BondySimon} proved in 1974 that for even cycles we have $ex(n,C_{2t})=O(n^{1+1/t})$.
Unfortunately, a matching lower bound is known only for the cases where $t=2,3,5$.
For a survey see~\cite{SimonSurvey, Verstraete}. 

In this paper we consider the case where $H = C_t$ and $t \coloneqq t(n)$ tends to infinity with $n$.
In this direction, it was proved by Erd\H{o}s and Gallai~\cite{ErdosGallai}, among other things, that if $t\coloneqq t(n)$, then $ex(n,P_t)=\floor{\frac 12(t-1)n}$.
For long cycles, it was shown by Woodall~\cite{Woodall} that if $t\geq \frac 12(n+3)$ then $ex(n,C_t)=\binom {t-1}2+\binom {n-t+2}2$, where the extremal example is given by two cliques intersecting in exactly one vertex.
In the same paper, Woodall also showed that for odd cycles $C_t$ shorter than $\frac 12(n+3)$, the trivial bound $ex(n, C_t) \geq \floor{\tfrac{n^2}4}$ is still tight.

In the past few decades several generalizations of the classical Tur\'an number $ex(n, H)$ were suggested and many results have been established in this area.
Denote by $ex(G,H)$ the number of edges in a largest subgraph of a graph $G$ containing no copy of $H$.
Note that the value of $ex(G,H)$ is bounded from below by the number of edges in $G$ that are not contained in any copy of $H$.
As a consequence, if the number of copies of $H$ in $G$ is much smaller than the number of edges in $G$, then we obtain that $ex(G,H)\geq (1-o(1))e(G)$.
Thus, it makes sense to restrict our attention to graphs $G$ for which the number of copies of $H$ is at least proportional to the number of edges. 

We focus on the case where the host graph $G$ is either a random graph or pseudo-random graph.
Given a positive integer $n$ and a real number $p \in [0, 1]$, we let  $G(n,p)$ be the \textit{binomial random graph}, that is, a graph sampled from the family of all labeled
graphs on the vertex set $[n] \coloneqq \{1,\dots,n \}$, where 
each pair of elements of $[n]$ forms an edge  with probability $p\coloneqq p(n)$, independently.
We denote by $ex(G(n,p),H)$ the number of edges in a largest subgraph of $G(n,p)$ without a copy of $H$ (note that $ex(G(n,p),H)$ is a random variable).
Clearly, in this case we want to consider only the values of $p$ for which $G(n,p)$ contains a copy of $H$ with high probability (w.h.p., i.e., with probability tending to 1 as $n\to \infty$), and in fact, the number of copies of $H$ in $G$ is typically ``large enough''.
   
 For fixed-size graphs $H$, this parameter has already been considered by various researchers.
 It is known that the threshold probability for a random graph to have the property that a typical edge is contained in a copy of $H$, for a fixed graph $H$, is $n^{-1/m_2(H)}$, where $m_2(H)$ is the \textit{maximum 2-density} and defined to be $m_2(H)= \max\left\{\frac{e(H')-1}{v(H')-2}\mid H'\subseteq H,\ v(H')\geq 3 \right\}$ (see \cite{FrBook} for more details).
 Therefore, it makes sense to consider graphs $G(n,p)$ for the regime $p = \Omega(n^{-1/m_2(H)})$.
 The cases $H=K_3$, $H=C_4$, and $H=K_4$ were solved by Frankl and R\"{o}dl~\cite{FranklRodlC3},  F\"{u}redi~\cite{FurediC4}, and by Kohayakawa, \L uczak, and R\"odl~\cite{KLRK4}, respectively.
 For fixed odd cycles, it was shown by Haxell, Kohayakawa, and \L uczak~\cite{HKLodd} that for $p\geq Cn^{-(2t-1)/2t}$ we have that $\frac 12 e(G(n,p))\leq ex(G(n,p),C_{2t+1})\leq \left(\frac 12+\varepsilon\right)e(G(n,p))$.
 For fixed even cycles, the same group of authors showed~\cite{HKLeven} that for $p=\omega(n^{-(2t-2)/(2t-1)})$ we have $ex(G(n,p),C_{2t})=o(e(G(n,p)))$  (for more precise bounds on the fixed even cycle case, see Kohayakawa, Kreuter, and Steger~\cite{KKSeven}, and Morris and Saxton~\cite{MSeven}).
 The authors of \cite{HKLodd,HKLeven, KLRK4}  conjectured that a similar behaviour should also hold for any fixed-size graph $H$, that is, that the value of $ex(G(n,p),H)$ should be asymptomatically equal to $\tfrac{ex(n,H)}{\binom n2} \cdot e(G(n,p))$, for suitable values of $p$.
 This conjecture was proved independently by Conlon and Gowers \cite{ConlonGowersTranseference} (with certain constraints on $H$) and by Schacht \cite{SchachtTransference}, who showed that the Tur\'an number of a fixed graph in $G(n,p)$ is of the same proportion of edges as it is in the complete graph, where the latter has been determined by Erd\H{o}s and Stone.
 More precisely, they proved that for $p\geq Cn^{-1/m_2(H)}$, and for a fixed graph $H$, \whp $ex(G(n,p),H)\leq (1-\frac 1{\chi(H)-1}+\varepsilon)e(G(n,p))$.
 A matching lower bound can be obtained by a random placement of the extremal example of $ex(n,H)$.   
 The phenomenon that we observe here is frequently called \textit{the transference principle}, which in this context can be interpreted as a random graph ``inheriting'' its (relative) extremal properties from the classical deterministic case, i.e., the complete graph.
 In their papers, Conlon and Gowers \cite{ConlonGowersTranseference} and Schacht \cite{SchachtTransference} discussed this principle and showed transference of several extremal results from the classical deterministic setting to the probabilistic setting.
 
In this paper we aim to study the transference principle in the context of long cycles.
The first step is to understand what should be the relevant regime of $p$.
It is easy to observe that if $p=o(\frac 1n)$ then a typical $G(n,p)$  is a forest, that is, does not contain any cycle.
Thus, when looking at the appearance of a cycle in $G(n,p)$, it is natural to restrict ourselves to the regime $p=\Omega (\frac 1n)$.
Furthermore, it is well known that cycles start to appear in $G(n,p)$ at probability $p=\Theta\left(\frac 1n\right)$.
We shall further recall what are the typical lengths of cycles one can expect to have in this regime. Note that for $p=\Theta\left(\frac 1n\right)$ \whp there are linearly many isolated vertices. Therefore, in this regime of $p$, we can hope to find in $G(n,p)$ cycles of length at most $(1-\varepsilon)n$ for some constant $\varepsilon>0$. Indeed, the typical appearance of nearly spanning cycles was shown in a series of papers by Ajtai, Koml\'os, and Szemer\'{e}di~\cite{AKSLongCycle}, de la Vega~\cite{DeLaVegaLongCycles}, Bollob\'{a}s\cite{BollobasLongCycle}, Bollob\'as, Fenner and Frieze~\cite{BFFLongCycles}.
In 1986 Frieze~\cite{FriezeLongCycles} proved that if $p\geq \frac Cn$ then \whp in $G(n,p)$ there exists a cycle of length at least $n-(1+\varepsilon)v^1(n,p)$, where $v^1(n,p)$ is the number of vertices of degree at most 1 and $\varepsilon\coloneqq \varepsilon(C)$ (and it was very recently improved even more by Anastos and Frieze~\cite{AnastosFrieze}).
In 1991, \L uczak showed~\cite{LuczakCycles} that for $p=\omega\left(\frac 1n\right)$, \whp $G(n,p)$ contains cycles of all lengths between $3$ and $n-(1+\varepsilon)v^1(n,p)$. 
On the other hand, when looking at cycles of length $o(\log n)$ in the context of Tur\'an-type problems, the regime $p=\Theta (\frac 1n)$ is not quite relevant.
It is easy to verify that for  $p=\Theta (\frac 1n)$ \whp one expects $o(e(G(n,p)))$ cycles of such lengths, and hence they can be destroyed by deleting a negligible proportion of edges.
Therefore, when requiring that the number of copies of $C_t$ will be \whp at least proportional to the number of edges, combining it with the fact that $p=\Omega\left(\frac 1n\right)$, we get that $t=\Omega(\log n)$.
 
 Moving back to the extremal problem, it was shown by  Dellamonica, Kohayakawa,  Marciniszyn, and  Steger~\cite{ResilienceLongCyclesRandomGraph} that if $p=\omega(\frac 1n)$, then for all $\alpha>0$, if $G'$ is a subgraph of $G(n,p)$ with $e(G')\geq \left(1-(1-w(\alpha))(\alpha+w(\alpha))+o(1)\right)e(G(n,p))$, then \whp $G'$ contains a cycle of length at least $(1-\alpha)n$, where $w(\alpha)=1-(1-\alpha)\lfloor(1-\alpha)^{-1} \rfloor $.
This result is asymptotically tight by the classical result of Woodall~\cite{Woodall} that guarantees a cycle of length at least $(1-\alpha)n$ in any graph $G$ with $e(G)\geq \left(1-(1-w(\alpha))(\alpha+w(\alpha))+o(1)\right)\binom n2$.

Very recently, Balogh, Dudek and Li \cite{BDL} studied the asymptotic behavior of $ex(G(n,p), P_{\ell})$ for various ranges of $\ell=\ell(n)$.

In this paper we study the appearance of long cycles of a \textbf{given length} in subgraphs of pseudo-random graphs.
As a direct consequence we get a result for $G(n,p)$.
More precisely, we determine the asymptotic value $ex((G(n,p)),C_t)$, where $p=\Omega(\frac 1n)$ and $t$ is between  $\Theta(\log n)$ and  $(1-\varepsilon)n$.

The more general statement deals with a class of graphs which is larger than the random graphs class.
For this we use the following definition.

\begin{definition}\label{def:UpperUni}
	Let $G$ be a graph on $n$ vertices.
	Suppose $0<\eta \leq 1$ and $0<p \leq 1$.
	We say that $G$ is \emph{$(p,\eta)$-upper-uniform} if for every $U,W \subseteq V(G)$ with $U\cap W = \emptyset$ and $|U|,|W| \geq \eta n$, we have $e_G(U,W) \leq (1+\eta)p|U||W|$.
\end{definition}

\begin{remark}\label{re:UpperUni}
	In a $(p,\eta)$-upper-uniform graph $G$ on $n$ vertices  we have, for any $U\subseteq V(G)$ with $|U|\geq 2\eta n$, that
	\[e(G[U]) \leq (1+\eta)p\binom{|U|}2. \]
\end{remark}

Indeed, let $U\subseteq V(G)$ of size $u \coloneqq |U|\geq 2\eta n$.
We look at all possible partitions of $U$ into two subsets $U_1, U_2$ such that $u_1\coloneqq |U_1| = \left\lfloor \tfrac u2 \right\rfloor$ and $u_2\coloneqq |U_2| = \left\lceil \tfrac u2 \right\rceil$, and we use it to count the number of edges in such cuts of $H$ in two ways.
We have
\begin{align*}
	e(U)\cdot 2\binom{u-2}{u_1-1} = \sum_{U_1, U_2} e(G[U_1, U_2]).
\end{align*}
By $(p,\eta)$-upper-uniformity of $G$ we have $e(G[U_1, U_2]) \leq (1+\eta)p u_1 u_2$, so we get
\begin{align*}
	e(U) &\leq \frac12 \binom{u-2}{u_1-1}^{-1} \binom{u}{u_1} (1+\eta)p u_1 u_2 = (1+\eta)p \binom u2.
\end{align*}

The following notation is based on results by Erd\H{o}s-Gallai~\cite{ErdosGallai} and Woodall~\cite{Woodall} (see \Cref{thm:ErdosGallai} and \Cref{thm:woodall} for more information).

\begin{definition}\label{def:gtn}
	 The functions $g_o, g_e$ are given as follows.
	 
	 If $t$ is odd, then
	 \begin{align*}
	 g_o\left( t, n \right)\cdot \binom n2 \coloneqq ex(n,C_t)+1= 
	 \begin{cases}
	 \binom{t-1}{2}+\binom{n-t+2}2+1, & \text{if } t\geq \frac 12(n+3),\\
	 \left\lfloor \frac 14n^2 \right\rfloor+1, & \text{if } t< \frac 12(n+3).
	 \end{cases}
	 \end{align*}
	 If $t$ is even and $\gamma>0$ is a parameter,
	 \begin{align*}
	 g_e^\gamma\left(t , n \right)\cdot \binom n2 \coloneqq 
	 \begin{cases}
	 ex(n,C_t)+1=
	 \binom{t-1}{2}+\binom{n-t+2}2+1, & \text{if } t\geq \frac 12(n+3),\\
	 ex(n,P_t)+1=\left\lfloor \frac 12 n (t-1)\right\rfloor+1, & \text{if } \gamma n\leq t< \frac 12(n+3),\\
	 0, & \text{if } t<  \gamma n,
	 \end{cases}
	 \end{align*}
	 Furthermore, the function $g^\gamma : [0,1] \rightarrow [0,1]$ is defined as follows.
	 \begin{align*}
	 g^\gamma(t, n) = \begin{cases}
	 g_o(t, n), & \text{if } t \text{ is odd}\\
	 g_e^\gamma(t, n), & \text{if } t \text{ is even}.\\
	 \end{cases}
	 \end{align*}
\end{definition}
Later we will set a specific value of the parameter $\gamma$ (see \Cref{parameters}).

We are now ready to state our main theorem. Here and later, $\log n$ refers to the natural logarithm.

\begin{theorem}\label{thm:main}
		 For every $0< \beta<\frac 14$, there exist $\eta, n_0,\gamma> 0$ such that for every $n\geq n_0$, if $G$ is a $(p,\eta)$-upper-uniform graph on $n$ vertices with $e(G) \geq (1-\beta/2)p\binom n2$ for some $0<p\coloneqq p(n)\leq1$, then for any $\frac {C_1}{\log(1/\beta)}\cdot \log n \leq t \leq (1-C_2\beta)n$, where $C_1,C_2>0$ are absolute constants, if $G'$ is a subgraph of $G$ with 
		\begin{align*}
			 e(G')\geq \left(g^\gamma(t, n) + \beta\right) e(G)
		\end{align*}
		edges, then $G'$ contains a cycle of length $t$.
\end{theorem}

Since we have $ex(n,P_t)=ex(n,P_{t-1})+O(n)$, then $ex(n,P_t)-O(n)\leq ex(n,P_{t-1})\leq ex(n,C_t)$, and
we can deduce from our main result the following corollary. 

\begin{corollary}\label{cor:main}
		For every $0< \beta<\frac 15$, there exist $\eta, n_0 > 0$ such that for every $n\geq n_0$, if $G$ is a $(p,\eta)$-upper-uniform graph on $n$ vertices with $e(G) \geq (1-\beta/2)p\binom n2$ for some $0<p\coloneqq p(n)\leq1$,  then for any $\frac {C_1}{\log(1/\beta)}\cdot \log n \leq t \leq (1-C_2\beta )n$, 
		\begin{align*}
		ex(G, C_t) \leq \left(\frac {ex(n,C_t)}{\binom n2} + \beta\right) e(G),
		\end{align*}
		where  $C_1,C_2>0$ are some absolute constants.
\end{corollary}

\begin{remark}\label{re:main}
	In both \Cref{thm:main} and \Cref{cor:main} we obtain, in fact, given $t$, all cycles of length $q$, where $\frac {C_1}{\log(1/\beta)}\cdot \log n \leq q\leq t$, with the same parity as $t$.
\end{remark}

\Cref{thm:main} and \Cref{cor:main} are asymptotically optimal in a stronger form; a matching lower bound is true for any graph $G$ on $n$ vertices, not only for upper-uniform graphs.
That is, for any graph $G$ on the vertex set $[n]$  there exists a subgraph $G_0$ with $\frac{ex(n,C_t)}{\binom n2}\cdot e(G)$ edges containing no cycle of length $t$.
Indeed, let $W_t$ be a graph on $n$ vertices with $ex(n,C_t)$ edges containing no cycle of length $t$.
By averaging, there exists an assignment $\sigma$ of the vertices of $W_t$ into $[n]$ such that when intersecting with $G$, we have $e(G\cap W^\sigma_t)\geq\frac {ex(n,C_t)}{\binom n2}\cdot e(G)$.
Clearly, the resulting graph $G\cap W^\sigma_t$ contains no cycles of length $t$.
This gives the following.
\begin{fact}
	For every graph $G$ on $n\geq 3$ vertices and every integer $t\in [3,n]$ we have
	\begin{align*}
	ex(G, C_t)\geq \frac{ex(n,C_t)}{\tbinom n2}e(G).
	\end{align*}	
\end{fact}

 \begin{remark}\label{re:pValue}
	\Cref{thm:main} does not assume anything on the value of $p$.
	However, graphs $G$ satisfying the conditions of the statement exist only for a restricted spectrum of values for $p$.
	More specifically, for $p = o(\tfrac1n)$ there are no $(p,\eta)$-upper-uniform graphs $G$ with $e(G) \geq (1-\beta/2)p\binom n2$, which makes the statement relevant only for $p\geq \tfrac Cn$ where $C>0$ is some constant.
	To see this, take $G$ to be a $(p,\eta)$-upper-uniform graph with $p=o(\frac 1n)$. Then by \Cref{re:UpperUni} $e(G)\leq (1+\eta)p\binom n2=o(n)$. Thus, there is a subset $I$ of isolated vertices in $G$ of size $\frac n2$. By the assumption $e(G)\geq (1-\beta)p\binom n2$ we obtain that $e(G[V\setminus I])\geq (1-\beta)p\binom n2$. This contradicts the upper uniformity of $G$ since $e(G[V\setminus I])\leq (1+\eta)p\binom {n/2}2$.
\end{remark}

Probably the most natural application of \Cref{thm:main} is for the random graph case. 
It is not hard to see that for $p=\Omega_\eta(1/n)$, $G(n,p)$  \whp satisfies the conditions of \cref{thm:main}.
Indeed, as the total number of edges in $G(n,p)$ is distributed binomially with parameters $\binom n2$ and $p$, we have that $e(G(n,p))\geq (1-\beta/2)p\binom n2$ with probability $1-e^{-\Omega(n)}$.
In addition, for, say, $p\geq \frac {\log 4}{\eta^4n}$ we have that for every two disjoint subsets $U_1,U_2$ such that $|U_1|,|U_2|\geq \eta n$, $e(U_1,U_2)\leq (1+\eta)p|U_1||U_2|$ with probability $1-e^{-\Omega(n)}$.
We obtain that, for $p \geq \tfrac Cn$ and $C\coloneqq C(\eta)$ being large enough, the random graph $G(n,p)$ is \whp $(p,\eta)$-upper-uniform. By the discussion regarding the expected cycle lengths in $G(n,p)$, we easily get that the lower bound on $t$ in \Cref{thm:main} is, in fact, necessary.

As a result, we obtain the following corollary.
\begin{corollary}\label{cor:Gnp}
	For every $0<\beta <\frac 14$, there exist $C,\gamma > 0$ such that if $G= G(n,p)$ where $p\geq \frac Cn$, then for any $\frac {C_1}{\log(1/\beta)}\cdot \log n \leq t \leq (1-C_2\beta )n$, with probability $1-e^{-\Omega(n)}$,
	where $C_1,C_2>0$ are absolute constants, if $G'$ is a subgraph of $G$ with  
	\begin{align*}
	e(G') \geq \left(g^\gamma(t, n) + \beta\right) e(G),
	\end{align*}
	then $G'$ contains a cycle of length $t$.
\end{corollary}

Similarly to \Cref{cor:main}, we can write the upper bound on $ex(G(n,p),C_t)$ only in terms of $ex(n,C_t)$, as follows.
\begin{corollary}
	For every $0<\beta <\frac 15$, there exists $C > 0$ such that if $G= G(n,p)$ where $p\geq \frac Cn$, then for any $\frac {C_1}{\log(1/\beta)}\cdot \log n \leq t \leq (1-C_2\beta )n$, with probability $1-e^{-\Omega(n)}$,
	\begin{align*}
	ex(G(n,p), C_t) \leq \left(\frac {ex(n,C_t)}{\binom n2} + \beta\right) e(G(n,p)),
	\end{align*}
	where $C_1,C_2>0$ are absolute constants.
\end{corollary}

Thus, also here, we observe a manifestation of the \emph{transference principle}, that is, the random graph $G(n,p)$ preserves the relative behavior of the Tur\'an number of long cycles observed in the classical case, i.e., in the complete graph $K_n$.

As mentioned in \Cref{re:main}, given $t$, the statement holds for every $\frac {C_1}{\log(1/\beta)}\cdot \log n \leq q\leq t$ with the same parity as $t$.\\

Another natural application of the main theorem is for $(n,d,\lambda)$-graphs, which can be shown to be $(p,\eta)$-upper-uniform for suitable values of $d, \lambda$.

\begin{definition}
	A graph $G$ is an \emph{$(n, d, \lambda)$-graph} if $G$ has $n$ vertices, is $d$-regular, and the second largest
	(in absolute value) eigenvalue of its adjacency matrix is bounded from above by $\lambda$. 
\end{definition}

$(n,d,\lambda)$-graphs have been studied extensively, mainly due to their good pseudo-random properties. For a detailed background see~\cite{KrivelevichSudakov}. 
Recently, it was shown  in~\cite{cyclesInExGraphs2} that for a given $\beta > 0$, if $\tfrac d\lambda \ge C(\beta)$, then $(n,d,\lambda)$-graphs contain cycles of all lengths between $\tfrac{C_1}{\log(1/\beta)}\cdot \log n$ and $(1 - C_2\beta)n$ (for some absolute constants $C_1,C_2>0$), improving the result in~\cite{cyclesInExGraphs1}.   

Using the Expander Mixing Lemma due to Alon and Chung \cite{ExpanderMixingLemma}, we can show that for suitable values of $d$ and $\lambda$, an $(n,d,\lambda)$-graph is also upper-uniform.
Hence we obtain the following corollary.

\begin{corollary}\label{cor:ndlambda}
	For every $0 < \beta < \tfrac14$ there exist $n_0, \gamma, \eta > 0$ such that for every $n \geq n_0$ and for every $d, \lambda > 0$ satisfying $\tfrac d\lambda \geq \tfrac1\eta$, 
	if $G$ is an $(n,d,\lambda)$-graph, then for any $\tfrac{C_1}{\log(1/\beta)}\cdot \log n \leq t \leq (1 - C_2\beta)$, where $C_1, C_2 > 0$ are absolute constants, if $G'$ is a subgraph of $G$ with
	\[e(G') \geq (g^\gamma(t, n) + \beta)e(G), \]
	edges, then $G'$ contains a cycle of length $t$.
\end{corollary}

Note that the lower bound on $t$ is tight due to the existence of $(n,d,\lambda)$-graphs with large girth. 
More explicitly, it was shown in \cite{ndlambdaGirthLub,ndlambdaGirthMar} 
that there exist infinitely many $(n, d, \lambda)$-graphs with girth $\Omega(\log n)$, such that $\tfrac d\lambda$ is larger than a given constant.
Details of the proof and further discussion on this application can be found in \Cref{sec:applications}.\\

Using very similar techniques, we can also obtain a \textit{robustness}-type result (for a detailed survey on robustness problems see~\cite{RobustnessSudakov}).
In this type of results, we consider a graph $G$ satisfying some \textit{extremal conditions} that guarantee a graph property $\mathcal P$ (in our case, containment of long cycles).
The aim is to measure quantitatively the strength of these specific conditions.
For this, we let $G$ be a graph satisfying these conditions, and let $G(p)$ be the random graph obtained by keeping each edge of $G$ independently with probability $p\in[0,1]$.
Note that if $G=K_n$ then $G(p)=G(n,p)$.
In the next theorem we show that if $G$ has (slightly more than) the minimum number of edges that guarantees a long cycle of a given length, then with high probability $G(p)$ also contains such a cycle for $p=\Omega\left(\frac 1n\right)$.
This value of $p$ is best possible due to threshold of the existence of cycles in $G(n,p)$.

\begin{theorem}\label{thm:mainRobustness}
	For every  $\beta > 0$ there exists $C > 0$ such that for  $\tfrac{C_1}{\log(1/\beta)}\log n \leq t \leq (1-C_2\beta)n$ (where  $C_1, C_2>0$ are absolute constants), for any $p\geq \tfrac Cn$, if $G$ is a graph on $n$ vertices satisfying
	\[ e(G) \geq ex(n, C_t) + \beta \binom n2, \]
	then \whp $G(p)$ contains a copy of $C_t$.
\end{theorem}

Note that starting with a graph with exactly $ex(n,C_t)+1$ edges is not enough.
Indeed, let $G$ be an extremal example for $ex(n, C_t)$ with an arbitrary edge $e$ added to it.
Then when taking $G(p)$ with $p=o(\frac 1n)$ \whp $e$ is deleted.
However, the above theorem shows that adding $\beta \binom n2$ edges to the extremal number will be enough, and, in fact, for many values of $t$ this number of edges is even tight.
We show that adding $\beta \binom n2$ edges to the extremal number is necessary in most cases.
However, there are values of $t$ for which only $\omega(n)$ extra edges to the extremal number suffice.
More precisely, this happens when $t < \tfrac12 n$ is odd, and recall that in this case we have $ex(n, C_t) = \floor{\tfrac{n^2}{4}}$.
This demonstrated in the following theorem (For more discussion see \Cref{sec:robust}).

\begin{theorem}\label{cl:TightOdd}
	For every  $\beta > 0$ there exists $C > 0$ such that for an odd $t$ with $\tfrac{C_1}{\log(1/\beta)}\log n \leq t \leq \left(\frac 12-\beta\right)n$ (where $C_1 > 0$ is an absolute constant), and for any $p\geq \tfrac Cn$, if $G$ is a graph on $n$ vertices satisfying
	\[ e(G) \geq ex(n, C_t) + \omega\left(\tfrac 1p \right), \]
	then \whp $G(p)$ contains a copy of $C_t$.
\end{theorem}

Finally, in \Cref{sec:Ramsey} we show some applications to Ramsey-type problems about cycles in random graphs. In particular, we use the power of the Key Lemma in order to argue about a typical appearance of a monochromatic cycle of a prescribed length in multicolored sparse random graphs.

\section{Notation and preliminaries}

Our graph-theoretic notation is standard, 
in particular we use the following. 
Let $[n]:=\{1,\dots , n\}$. For a positive real number $\ell$, we denote by $\floor{\ell}_{odd}$ (respectively, $\floor{\ell}_{even}$) the largest odd (respectively, even) integer $m$ with $m\leq \ell$.

For a graph $G = (V,E)$ and a set $U\subset V$, let $G[U]$ denote the corresponding vertex-induced subgraph of $G$. We also denote $e(G)=|E(G)|$ and $v(G)=|V(G)|$. For $U\subset V$ we let $\Gamma_G(U)=\{v\in V\setminus U \mid \exists u\in U\text{ s.t. }\{u,v\}\in E  \}$ be the neighborhood of $U$ in $G$.
For an integer $k$ and $V_i\subseteq V$, $i\in [k]$, we say that $\Pi=(V_1,\dots,V_k)$ is a \textit{partition} of $V$ if $V=\bigcup_{i\in [k]}V_i$ and $V_i\cap V_j=\emptyset$ for every $i\neq j$. 

\subsection{Known extremal results}

To prove our result, we use two classical theorems, one by Woodall~\cite{Woodall} about cycles, and the other one by Erd\H{o}s and Gallai~\cite{ErdosGallai} regarding paths.

\begin{theorem}[\cite{ErdosGallai}, Theorem 2.6]\label{thm:ErdosGallai}
	Let $G$ be an $n$-vertex graph with more than $\left\lfloor\frac 12 n (t-1)\right\rfloor$ edges.
	Then $G$ contains a path of length at least $t$ (the number of edges). 
\end{theorem}

By looking at a graph consisting of $\left\lfloor\frac {n}{t}\right\rfloor$ vertex-disjoint cliques of size $t$ and another clique on the remaining vertices, one can observe that the above result is tight.

\begin{theorem}[\cite{Woodall}, Corollary 11]\label{thm:woodall}
	Let $G$ be a graph on $n\geq 3$ vertices and let $3\leq t\leq n$. Assume that $e(G)\geq w(t,n)\cdot \binom n2$ where 	
	\begin{align*}
	w\left(t,n \right)\cdot \binom n2 \coloneqq 
	\begin{cases}
	\binom{t-1}{2}+\binom{n-t+2}2+1, & \text{if } t\geq \frac 12(n+3),\\
	\left\lfloor \frac 14n^2 \right\rfloor+1, & \text{if } t< \frac 12(n+3).
	\end{cases}
	\end{align*}
Then $G$ contains a cycle of length $d$ for any $3\leq d\leq t$.  

\end{theorem}

The result in \Cref{thm:woodall} is tight in the sense that there are graphs with $w(t,n)\binom n2-1$ edges containing no cycle of some length between $3$ and $t$, and corresponding extremal examples are constructed explicitly (as was mentioned briefly in the beginning of the introduction).
For $t\geq \frac 12(n+3)$, the graph consisting of two cliques, one of size $\binom{t-1}2$ and the other of size $\binom{n-t+2}2$, sharing exactly one vertex, does not contain a cycle of length $t$ or longer.
For $t< \frac 12(n+3)$, the complete bipartite graph with $\left\lfloor \frac 14n^2 \right\rfloor$ edges does not contain a cycle of any odd length, and in particular of any odd length between $3$ and $t$.

Note that the function $w(t, n)$ of Woodall is strongly related to the function $g^\gamma(t, n)$ given in \Cref{def:gtn}.
In particular, for odd values of $t$ we have $w(t,n) = g_o(t,n)$, and furthermore, $w(t,n)\geq \frac 12$ for any $t$ and $n$.
In addition, $w(t,n)$ is monotone increasing in $t$ for this case.
So for any odd $t$, $w(t,n)\tbinom n2=g_o(t,n)\tbinom n2=ex(n,C_t)-1$.
As for even values of $t$, we get $w(t,n) = g_e(t,n)$ only when $t\geq \tfrac12(n+3)$.
For an even $t< \frac 12(n+3)$, note that we \textbf{do not} necessarily have $w(t, n)\binom n2 = ex(n, C_t) + 1$ (although, as mentioned, \Cref{thm:woodall} is still tight because of the requirement of having all cycles, also the odd ones, of length at most $t$.)
For this reason, we also make use of $ex(n,P_t)$ in \Cref{def:gtn} for even values of $t < \frac 12(n+3)$. 

\begin{remark}\label{re:gtn}
	Note that if $0<\varphi<\frac 12$ is constant and $t=(1-\varphi+o_n(1)) n$ then $w(t,n)=1-2\varphi+2\varphi^2+o_n(1)$. In particular, if $e(G)\geq (1-\varphi)\binom n2$, then $G$ contains a cycle of length $d$ for any $3\leq d\leq (1-\varphi)n$.
\end{remark}

Another result to be used in this paper in a significant way is by Friedman and Pippenger~\cite{FP}, regarding the existence of large trees in expanding graphs.
\begin{theorem}[\cite{FP}, Theorem 1]\label{FP}
	Let $T$ be a tree on $k$ vertices of maximum degree at most $d$.
	Let $H$ be a non-empty graph such that, for every $X\subset V(H)$ with $|X| \leq 2k-2$ we have $|\Gamma_H(X)| \geq (d+1)|X|$.
	Let further $v\in V(H)$ be an arbitrary vertex of $H$.
	Then $H$ contains a copy of $T$, rooted at $v$.
\end{theorem}

\subsection{Sparse Regularity Lemma}\label{sec:SRL}
In order to prove~\Cref{thm:main}, we make use of a variant of Szemer\'edi’s Regularity Lemma~\cite{RegularityLemma} for sparse graphs, the so-called
Sparse Regularity Lemma due to Kohayakawa~\cite{SRL} and R\"odl (see~\cite{ConlonSRL,GerkeSteger,KRSparse}).
The sparse version of the Regularity Lemma is based on the following definition.
\begin{definition}\label{def:reg}
	Let a graph $G = (V,E)$ and a real number $p\in(0,1]$ be given.
	We define the \emph{$p$-density} of a pair of non-empty, disjoint sets $U,W\subseteq V$ in $G$ by
	\[d_{G,p}(U,W) = \frac{e_G(U,W)}{p|U||W|}. \]
	For any $0<\varepsilon\leq 1$, the pair $(U,W)$ is said to be \emph{$(\varepsilon,G,p)$-regular}, or just \emph{$(\varepsilon,p)$-regular} for short, if, for all $U'\subseteq U$ with $|U'|\geq \varepsilon|U|$ and all $W'\subseteq W$ with $|W'|\geq \varepsilon |W|$, we have
	\begin{align}\label{eq:regularity}
	\left|d_{G,p}(U,W) - d_{G,p}(U',W') \right| \leq \varepsilon.
	\end{align}
	
	We say that a partition $\Pi = (V_1, \ldots, V_k)$ of $V$ is \emph{$(\varepsilon, p)$-regular} if $\left||V_i| - |V_j|\right| \leq 1$ for all $i, j \in [k]$, and, furthermore, at least $(1-\varepsilon)\binom k2$ pairs $(V_i, V_j)$ with $1\leq i < j \leq k$ are $(\varepsilon, p)$-regular.
	
	In the case $p=1$ we say that the pair (or the partition) is $\varepsilon$-regular.
\end{definition}

\begin{theorem}[Sparse Regularity Lemma \cite{SRL}]\label{thm:srl}
	For any given $\varepsilon>0$ and $k_0\geq 1$, there are constants $\eta = \eta(\varepsilon, k_0) > 0$ and $K_0 = K_0(\varepsilon, k_0) \geq k_0$ such that any $(p,\eta)$-upper-uniform graph $G$ on $n$ vertices, for large enough $n$, with $0<p\leq 1$ admits an $(\varepsilon, p)$-regular partition of its vertex set into $k$ parts, where $k_0\leq k \leq K_0$.
\end{theorem}

\begin{remark}\label{parameters}
	In the main proof we make an extensive use of the Sparse Regularity Lemma (and, in fact, also of the Regularity Lemma in \Cref{sec:robust}). As a result, we need to keep many parameters in mind. For simplicity, we present here some of the parameters and the relations between them. Unless mentioned otherwise, these are the values of the parameters during the proofs in the next sections, given here for future reference:\\
	
	\begin{tabular}{ l l l }
	$\varepsilon\leq \tfrac \beta{10000}$ & regularity parameter\\
	$\rho = 10\varepsilon$ & density parameter\\
	$k\geq k_0\geq \frac 2{\varepsilon^2}$ & number of clusters\\
	$\eta\leq \min(\frac{1}{3K_0}, \eta^*)$ & parameter of upper uniformity\\
	$\tau = \frac \beta{32}$ & ``extra" number of edges we have in the reduced graph\\
	$\delta=48\varepsilon$ & proportion of number of vertices we are not able to use in each cluster\\
	$\gamma\leq \frac {2(1-48\varepsilon)}k$ & parameter of $g^\gamma_e(t,n)$ and $g^\gamma(t,n)$ that appears in \Cref{def:gtn} and in \Cref{thm:main}\\
	$m=\frac nk$ & size of each cluster up to $\pm1$
	\end{tabular}
	where $K_0$ and $\eta^*$ are as given in \Cref{thm:srl} (taking $\eta^*$ to be $\eta$).
\end{remark}

\subsection{Organization}
As was mentioned before, in the proof of Theorem~\ref{thm:main} we rely heavily on the Sparse Regularity Lemma (see \Cref{sec:SRL}). Roughly speaking, we use the lemma to obtain a regular partition of our graph into clusters. Then, we define an auxiliary graph (the \textit{reduced graph}) in which each vertex represents a cluster of the original graph, and show that if this auxiliary graph has enough edges, then the original graph contains the desired cycle. For this, in~\Cref{sec:reduced} we define the Reduced Graph and prove that it contains many edges.  Then, in \Cref{sec:key} we present the Key Lemma used in the paper to convert a cycle in the reduced graph to a cycle of an appropriate length in the original graph. In the same section we give the proof of~\Cref{thm:main} using the Key Lemma. \Cref{sec:proofKey} is devoted for the proof of the Key Lemma. In~\Cref{sec:extra} we give some further related results.  

\section{The Reduced Graph}\label{sec:reduced}
\begin{definition}[Reduced Graph]\label{def:RGraph}
	Let $\varepsilon>0$, $k\geq 1$ an integer, $0\leq p\leq1$, and $0 < \rho \leq 1$.
	Let $G_0$ be a graph on $n$ vertices, and $\Pi=(V_1,\dots,V_k)$ a partition of its vertices.
	We define the \emph{reduced graph} $R(G_0, \Pi, \rho, \varepsilon, p)$ to be the graph on the vertex set $\{1, \ldots, k\}$, where vertices $i$ and $j$ are connected by an edge if and only if $(V_i, V_j)$ is $(\varepsilon,p)$-regular and $d_{G_0,p}(V_i, V_j) \geq \rho$.
	If we consider the reduced graph where $p = 1$, we omit this parameter from the notation.  
\end{definition}

\begin{lemma}\label{lem:EdgesRGraph}
	Let $0 < \beta < \tfrac14$, $x\in [0,1)$ such that $x+\beta<1$.
	Let $\varepsilon \leq \tfrac \beta {1000}$, $k\geq \tfrac{100}\beta$, $\tau = \tfrac \beta {32}$, and $\eta \leq \tfrac1{3k}$ be positive.
	Assume that $G$ is an $(p,\eta)$-upper uniform graph, and $e(G) \geq (1-\beta/2)p\binom n2$, for some $0 < p \coloneqq p(n) \leq 1$.
	Let $G'$ be the graph obtained from $G$ by keeping at least $\left(x + \beta \right)e(G)$ edges, and assume that $\Pi=(V_1,\dots,V_k)$ is an $(\varepsilon,p)$-regular partition of $G'$.
	Let $R\coloneqq R(G',\Pi, \rho, \varepsilon ,p)$ be the reduced graph as in \Cref{def:RGraph} for $\rho = 10\varepsilon$.
	Then 
	\[e(R) \geq \left(x + \tau \right) \binom k2. \]
\end{lemma}

\begin{proof}
	Denote $m = \tfrac nk$, and recall that $\lfloor m \rfloor \leq |V_i| \leq \lceil m \rceil$ for any $i\in [k]$.
	Now we count the number of edges of $G'$.
	
	\begin{itemize}
		\item The number of edges with endpoints in the same $V_i$ for $1\leq i \leq k$ is at most
		\begin{align*}
			k(1+\eta)p\binom{\ceil{m}}2 \leq \frac 2k (1+\eta)^2 \frac {pn^2}2
		\end{align*}
		
		\item The number of edges in irregular pairs is at most
		\begin{align*}
			\varepsilon \binom k2 (1+\eta)p\ceil{m}^2 \leq \varepsilon (1+\eta)^2 \frac{pn^2}2.
		\end{align*}
		
		\item The number of edges in pairs that are of $p$-density less than $\rho$ is at most
		\begin{align*}
			\binom k2 \rho p \ceil{m}^2 \leq (1+\eta)\rho \frac{pn^2}2.
		\end{align*}
		
		\item The number of edges in $(\varepsilon, p)$-regular pairs $(V_i, V_j)$ with $p$-density at least $\rho$ is at most
		\begin{align*}
		e(R) (1+\eta)p\ceil{m}^2 \leq e(R)(1+\eta)^2\frac{2}{k^2}\cdot\frac{pn^2}2.
		\end{align*}
	\end{itemize}

	In total we get
	\begin{align*}
	e(G') \leq(1+\eta) \left((1+\eta)\left(\frac{2}{k^2}e(R) + \varepsilon + \frac 2k \right) + \rho \right)\frac{pn^2}2.
	\end{align*}
	
	On the other hand, recall that 
	\begin{align*}
		e(G') &\geq \left(x + \beta \right)e(G) \geq \left(x + \beta \right)(1-\beta/2)p\binom n2 \\
		&\geq \left(x + \beta \right)(1 - 2\beta/3)\frac{pn^2}2,
	\end{align*}
	so we get
	\begin{align*}
	(1+\eta)\left((1+\eta)\left(\frac{2}{k^2}e(R) + \varepsilon + \frac 2k \right) + \rho \right)\frac{pn^2}2 \geq \left(x + \beta \right)(1 - 2\beta/3)\frac{pn^2}2,
	\end{align*}
	and hence
	\begin{align*}
	e(R) > \left(\frac{\left(x + \beta \right)(1 - 2\beta/3) - \rho}{(1+\eta)^2} - \varepsilon - \frac2k\right)\frac{k^2}2 \geq \left(x + \tau \right)\binom k2,
	\end{align*}
	where the last inequality follows by the choice of the parameters $\varepsilon, \eta, \rho, k, \tau$ combined with the fact that $x\leq 1$.
\end{proof}

\section{Key Lemma and proof of Theorem \ref{thm:main}}\label{sec:key}

In this section we state the Key Lemma and then use it to prove \Cref{thm:main}.

\begin{definition}
	Let $G$ be a graph and let $V_1, V_2 \subseteq V(G)$ be two disjoint subsets of vertices with $|V_1|, |V_2| \in \{\floor{m}, \ceil{m} \} $ for some positive number $m$.
	Let $\varepsilon > 0$.
	We say that the pair $(V_1, V_2)$ satisfies the \emph{$\varepsilon$-property} in $G$ if for every two subsets $U_1\subseteq V_1$ and $U_2\subseteq V_2$ with $|U_1|, |U_2| \geq \varepsilon m$, $G$ contains at least one edge between them, i.e., $e(G[U_1, U_2]) > 0$. 
	
\end{definition}

\begin{definition}\label{def:SGraph}
	Let $\varepsilon > 0$ and let $k$ be a positive integer.
	Let $G_0$ be a graph and let $\Pi = (V_1,\dots,V_k)$ be a partition of its vertices into $k$ parts satisfying $\left||V_i| - |V_j| \right| \leq 1$ for all $i,j \in [k]$.
	We define the $\varepsilon$-graph $S\coloneqq S(G_0,\Pi, \varepsilon)$ to be the graph with vertex set $[k]$ where $\{i, j\} \in E(S)$ if the pair $(V_i, V_j)$ satisfies the $\varepsilon$-property in $G_0$. 
\end{definition}

\begin{lemma}[Key Lemma]\label{lem:key}
	Let $0 < \varepsilon < \tfrac1{85}$.
	Let $G_0$ be a graph on $n$ vertices, for large enough $n$, and let $\Pi = (V_1, \ldots, V_k)$ be a partition of its vertices satisfying $\left||V_i| - |V_j| \right| \leq 1$ for all $i, j \in [k]$, where $\tfrac{2}{\varepsilon^2} \leq k $ is a constant.
	Let $S\coloneqq S(G_0,\Pi, \varepsilon)$ be the corresponding $\varepsilon$-graph as in~\Cref{def:SGraph}.
	Then for $\delta = 48\varepsilon$ and any absolute constant $C_1>2.1$ we have the following.
	\begin{itemize}
		\item If $S$ contains a path of an odd length $b$, $1 \leq b < k$, then $G_0$ contains cycles of all even lengths in $\left[\frac {C_1}{\log (1/\varepsilon)}\log n, (1-\delta)a n\right]$, with $a \coloneqq \tfrac {b+1}k$.
		\item If $S$ contains a cycle of an odd length $b$, $3 \leq b < k$, then $G_0$ contains cycles of all odd lengths in $\left[\frac {(b-1)C_1}{2\log (1/\varepsilon)}\log n, (1-\delta)a n\right]$, with $a\coloneqq \tfrac bk$.
	\end{itemize}
\end{lemma}

The assumption in the first item that $b$ is odd is of technical nature and is in fact an artifact of our proof strategy. The proof of the Key Lemma can be found in \Cref{sec:PfKey}.

Using this Key Lemma, we can deduce the existence of long cycles in a graph in cases where there are enough edges in a corresponding $\varepsilon$-graph.

\begin{corollary}\label{cor:key}
	Let $0 < \beta < 1/3$, $\varepsilon = \frac {\beta}{10000}$, $k\geq \frac 2{\varepsilon^2}$, $\gamma \leq \tfrac{2(1-48\varepsilon)}k$, and $\delta = 48\varepsilon$.
	Let $G_0$ be a graph on $n$ vertices, for large enough $n$, and let $\frac {C_1}{\log(1/\beta)}\cdot \log n \leq t \leq (1-C_2\beta )n$, where $C_2 \geq \tfrac{48}{10000}$ is an absolute constant and $C_1$ is the absolute constant from \Cref{lem:key}.
	Assume that there exists a partition $\Pi=(V_1,\dots,V_k)$ of the vertices of $G_0$ such that the corresponding $\varepsilon$-graph $S\coloneqq S(G_0,\Pi, \varepsilon)$ satisfies $e(S) \geq (g^\gamma(t,n)+\beta/32)\binom k2$, where $g^\gamma(t, n)$ is defined in \Cref{def:gtn}.
	Then $G_0$ contains a cycle of length $t$.
\end{corollary}  

\begin{proof}
	We split the proof into four cases by the parity and the value of $t$.
	Throughout all following cases we use the facts that $1 - C_2\beta \leq 1 - \delta$ and that $\varepsilon<\beta$.
	
	\textbf{Case 1:} $t$ is even and $t < \gamma n$.
	In this case we have $g^\gamma_e(t,n) = 0$, and in particular $e(S) \geq \tfrac \beta{32} \binom k2$.
	By \Cref{thm:ErdosGallai} we get that $S$ contains a path of length at least $\tfrac \beta{32}\cdot k > 1$ and hence, by \Cref{lem:key}, $G_0$ contains a cycle of length $t$. 
	
	\textbf{Case 2:} $t$ is even and $\gamma n \leq t < \tfrac12(n+3)$.
	In this case we have $g^\gamma_e(t,n) = \tfrac{\floor{\tfrac12 n(t-1)} + 1}{\binom n2}$, and in particular $e(S) \geq \left(\tfrac{\floor{\frac12 n(t-1)} + 1}{\binom n2} + \tfrac\beta{32} \right)\binom k2 \geq \left(\tfrac tn + \tfrac\beta{50} \right) \binom k2$.
	By \Cref{thm:ErdosGallai} we get that $S$ contains a path of length at least $\left(\tfrac tn + \tfrac\beta{50} \right)k$ and hence, by \Cref{lem:key}, and since $t< \left(\frac tn+\frac\beta{50} \right)(1-\delta) n$, $G_0$ contains a cycle of length $t$.
	
	\textbf{Case 3:} $t$ is odd and $t < \tfrac12(n+3)$.
	In this case we have $g_o(t,n) = \tfrac{\floor{\tfrac14 n^2} + 1}{\binom n2}$, and in particular $e(S) \geq \left(\tfrac{\floor{\tfrac14 n^2} +1}{\binom n2} + \tfrac \beta{32} \right)\binom k2>\left(\frac 12+\frac \beta{32}\right)\binom k2$.
	By \Cref{re:gtn} and \Cref{thm:woodall} we get that $S$ contains cycles of all lengths up to
	$\left(\frac 12+\frac \beta{32}\right)k>\left(\tfrac {n+3}{2n} + \tfrac \beta{50} \right)k>\left(\tfrac tn + \tfrac \beta{50} \right)k$ 
	and hence, by \Cref{lem:key}, $G_0$ contains a cycle of length $t$. 
	
	\textbf{Case 4:} $\tfrac12(n+3) \leq t \leq (1-C_2\beta)n$.
	In this case we have $g^{\gamma}(t,n) = \tfrac{\binom{t-1}2 + \binom{n-t+2}2 + 1}{\binom n2}$, and thus $e(S) \geq \left(\frac{\binom{t-1}2 + \binom{n-t+2}2 +1}{\binom n2} + \frac \beta{32} \right)\binom k2\geq \frac tn(1+\frac {3\delta}{1-\delta})\binom k2$.
	By \Cref{re:gtn} and \Cref{thm:woodall} we get that $S$ contains cycles of all lengths up to $\tfrac tn\left(1 + \tfrac {3\delta}{1-\delta} \right)k$ and in particular a path of length, say, $\left\lceil\tfrac tn\left(1 + \tfrac {2\delta}{1-\delta} \right)k\right\rceil$.
	By \Cref{lem:key}, since $t\leq \tfrac tn\left(1 + \tfrac {2\delta}{1-\delta} \right)(1-\delta) n-1$, we get that $G_0$ contains a cycle of length $t$, where if $t$ is odd then we look at the cycle in $S$ and if $t$ is even then we look at the path.
\end{proof}

\begin{remark}\label{re:gmonotonicity}
	Given $\gamma$, note that the function $g^\gamma(t,n)$ is monotone in the following sense.
	For any $0<t<\frac14(n+3)$ we have $g^\gamma(2t+1,n)\geq g^\gamma(2t,n)$, $g^\gamma(2t+1,n)\geq g^\gamma(2t-1,n)$, and $g^\gamma(2t+2,n)\geq g^\gamma(2t,n)$.
	In addition, if $t\geq\frac 12(n+3)$ then $g^\gamma(t+1,n)\geq g^\gamma(t,n)$.
	Consequently, under the assumptions of \Cref{cor:key}, if $t$ is odd then $G$ contains \textbf{all} cycles of lengths between $\frac {C_1}{\log(1/\beta)}\log n$ and $t$, and if $t$ is even  then $G$ contains \textbf{all even} cycles of lengths between $\frac {C_1}{\log(1/\beta)}\log n$ and $t$.
	In addition, if $t\geq\frac 12(n+3)$ then $G$ contains \textbf{all} cycles of lengths between $\frac {C_1}{\log(1/\beta)}\log n$ and $t$ (regardless of the parity of $t$).
\end{remark}

We next show that $p$-regular pairs of subsets in our graph with non-negligible $p$-density satisfy the $\varepsilon$-property.
Then by the Key Lemma we can deduce the main theorem. 

\begin{claim}\label{cl:OneEdge}
	Let $n$ be an integer,  $\varepsilon>0$, 
	$\varepsilon<\rho<\frac 12$.
	Let $G_0$ be a graph on $n$ vertices and let $V_1, V_2 \subseteq V(G_0)$ be two subsets of vertices satisfying: $V_1 \cap V_2 = \emptyset$, $|V_1|, |V_2| \in \{\floor{m},\ceil{m} \}$ for some $m$, and the pair $(V_1, V_2)$ is $(\varepsilon, p)$-regular in $G_0$ with $ d_{G_0,p}(V_1, V_2) \geq \rho$, for some $0 < p \coloneqq p(n) \leq 1$. 
	Then the pair $(V_1,V_2)$ satisfies the $\varepsilon$-property in $G_0$.
\end{claim}

\begin{proof}
	Let $U_1\subseteq V_1$ and $U_2\subseteq V_2$ be such that $|U_1|, |U_2| \geq \varepsilon m$.
	By regularity we have\\ $\left|d_{G_0,p}(V_1,V_2) - d_{G_0,p}(U_1,U_2) \right| \leq \varepsilon$.
	Combining it with the assumption $d_{G_0,p}(V_1, V_2) \geq \rho$, we have that
	\[e(U_1,U_2)\geq (\rho-\varepsilon)p|U_1||U_2|>0.\qedhere \]
\end{proof}

Using \Cref{cor:key} we can immediately prove our main theorem.

\begin{proof}[Proof of Theorem~\ref{thm:main}]
	Let $\varepsilon = \frac {\beta}{10000}$, $\rho = 10\varepsilon$, and $k_0 = \frac{2}{\varepsilon^2}$.
	Let $\eta_0\coloneqq \eta_0(\varepsilon,k_0)>0$, $K_0\coloneqq K_0(\varepsilon,k_0)\geq k_0$, and $k \in [k_0, K_0]$ be as given by the Sparse Regularity Lemma (Theorem~\ref{thm:srl}) applied with $\varepsilon$ and $k_0$.
	Let $\eta\coloneqq \min \{\eta_0,\tfrac 1{3k_0}\}$ and let $\gamma = \tfrac{2(1-48\varepsilon)}k$.
	Recall that $G$ is a $(p,\eta)$-upper-uniform graph for some $0 < p \coloneqq p(n) \leq 1$, with $e(G) \geq (1-\beta/2)p\binom n2$.
	Let $G'$ be a graph obtained from $G$ by keeping at least $(g^\gamma(t,n)+\beta)e(G)$ edges, and note that $G'$ is also $(p,\eta)$-upper-uniform.
	Let $\Pi=(V_1,\dots,V_k)$ be an $(\varepsilon,p)$-regular partition of $G'$ guaranteed by the Sparse Regularity Lemma for the relevant parameters, for some $k_0 \leq k\leq K_0$.
	Let $R\coloneqq R(G',\Pi, \varepsilon, \rho, p)$ be the reduced graph on $k$ vertices with parameters $\rho,\varepsilon, p$ and $k$, as in \Cref{def:RGraph}.
	By \Cref{cl:OneEdge}, if $\{i,j\}$ is an edge in $R$, then the pair $(V_i,V_j)$ satisfies the $\varepsilon$-property in $G'$.
	Hence, the reduced graph $R$ is a subgraph of the $\varepsilon$-graph $S\coloneqq S(G',\Pi, \varepsilon)$, as defined in \Cref{def:SGraph}.
	In particular $e(S)\geq e(R)$, and every path or cycle contained in $R$ is also contained in $S$. 
	 
	Let $C_1 > 0$ be the constant from \Cref{lem:key}, and let $C_2>0$ be the constant from \Cref{cor:key}.
	Let $\tfrac{C_1}{\log(1/\beta)}\log n\leq t\leq (1-C_2\beta)n$.
	By Lemma~\ref{lem:EdgesRGraph} we have that $e(R)\geq(g^\gamma(t, n)+\tau)\binom k2$ (where $\tau=\frac \beta{32}$), and thus $e(S)\geq(g^\gamma(t, n)+\tau)\binom k2$. 
	Applying \Cref{cor:key}, we get a cycle of length $t$ in $G$.
\end{proof}

Applying \Cref{re:gmonotonicity} to \Cref{thm:main}, note that if $t$ is odd then $G$ contains \textbf{all} cycles of lengths between $\frac {C_1\log n}{\log(1/\beta)}$ and $t$, and if $t$ is even  then $G$ contains \textbf{all even} cycles of lengths between $\frac {C_1\log n}{\log(1/\beta)}$ and $t$. In addition, if $t>\frac 12(n+3)$ then $G$ contains \textbf{all} cycles of lengths between $\frac {C_1\log n}{\log(1/\beta)}$ and $t$ (regardless of the parity of $t$).

\section{Proof of the Key Lemma}\label{sec:proofKey}
In this section we prove the Key Lemma (\Cref{lem:key}) using several claims and results regarding tree embeddings in expander graphs.
The main idea is to show that every two vertices connected by an edge in the reduced graph represent a pair of clusters in the original graph that has ``good expansion" properties (\Cref{sec:expander}).
Then, we show that the graph induced by any pair of such clusters contains a very specific tree (\Cref{sec:tree}), which will later be used to embed the desired cycle (\Cref{sec:PfKey}).

\subsection{Expander graphs}\label{sec:expander}

\begin{definition}\label{def:expander}
	A graph $G = (V,E)$ is called a \emph{$(B, \ell)$-expander} if for every $X\subseteq V$ with $|X| \leq B$ we have $|\Gamma_G(X)| \geq \ell |X|$.
\end{definition}

For the proofs in this section we also need a somewhat more specific definition of expander graphs for the special case of bipartite graphs.

\begin{definition}\label{def:bipexpander}
	A bipartite graph $G=(V_1\cup V_2,E)$ is called a \emph{$(B,\ell)$-bipartite-expander} if for every $X\subseteq V_i$, $1\leq |X|\leq B$, we have $|\Gamma_G(X)|\geq \ell|X|$.
\end{definition}

\begin{remark}\label{rem:expander}
	If a bipartite graph $G$ is an $(A, \ell+1)$-bipartite-expander, then it is a $\left(2A, \tfrac12\ell \right)$-expander.
\end{remark}

\begin{proposition}\label{prop:expander}
	Let $\varepsilon > 0$ and let $a, b > 0$ satisfy $(2b + 2)(1 - \varepsilon - ab) > 1$ and $(2b + 2)\varepsilon \geq 1$. 
	Let $G$ be a bipartite graph with parts $V_1, V_2$ with $|V_1|, |V_2| \geq (2b + 2)\varepsilon m$ for some integer $m$, and assume that every two subsets $V'_1, \subseteq V_1$, $V'_2\subseteq V_2$ with $|V'_1|, |V'_2| \geq \varepsilon m$ span at least one edge in $G$, i.e., $e(G[V'_1, V'_2]) > 0$.
	Then there exist $U_1\subseteq V_1$ and $U_2\subseteq V_2$ with $|U_1| \geq (1-\varepsilon)|V_1|$ and $|U_2| \geq (1-\varepsilon)|V_2|$ such that the bipartite graph $G[U_1, U_2]$ is an $(ax, b)$-bipartite-expander, where $x = \min(|V_1|, |V_2|)$.
\end{proposition}
  
\begin{proof}
	If every subset of $X_i\subseteq V_i$ of size at most $ax$ satisfies $|\Gamma_G(X_i)| \geq b |X_i|$ then we are done by setting $U_1 = V_1$ and $U_2 = V_2$.
	Otherwise, there are subsets violating the expansion condition.
	We iteratively remove such subsets of size at most $\varepsilon m$, one by one, to create an $(\varepsilon m,b)$-bipartite-expander. We then show that the expander we have created is, in fact, an $(ax,b)$-bipartite-expander.
	More formally, we define $V_1^0 = V_1$, $V_2^0 = V_2$ and $W_1^0 = W_2^0 = \emptyset$.
	Let $r\in \mathbb N\cup \{0\}$.
	If for $1\leq i\neq j \leq 2$, there exists $W\subset V_i^r$ with $|W| \leq \varepsilon m$ and $|\Gamma(W)\cap V_j^r| < b |W|$, then we define $V_i^{r+1} = V_i^r\setminus W$, $V_j^{r+1} = V_j$, and $W_i^{r+1} = W_i^r \cup W$, $W_j^{r+1} = W_j^r$.
	If at some point $r_0$ there are no more subsets violating the $(\varepsilon m,b)$-expansion condition in $V_1^{r_0}, V_2^{r_0}$, and we have $|W_1^{r_0}|, |W_2^{r_0}| < \varepsilon m$, then we define $U_1 = V_1^{r_0}$, $U_2 = V_2^{r_0}$, which means that the graph $G[U_1, U_2]$ is an $(\varepsilon m, b)$-bipartite-expander.
	Otherwise, for some $r_0$ we have, for the first time in this process, $|W_i^{r_0}| \geq \varepsilon m$ for some $i\in \{1,2\}$.
	Since in each step $r$ of the process we add to one of $W_1^{r-1}, W_2^{r-1}$ at most $\varepsilon m$ vertices, it follows that $\varepsilon m \leq |W_i^{r_0}| \leq 2\varepsilon m$.
	By the definition of $W_i^{r_0}$ we get $|\Gamma(W_i^{r_0})\cap V_j^{r_0}| < b |W_i^{r_0}|$, where $i \neq j \in \{1,2\}$.
	By the choice of $r_0$ we know that $|W_j^{r_0}| < \varepsilon m$ ($j\neq i$), and thus
	\[|V_j^{r_0} \setminus \Gamma(W_i^{r_0})| > |V_j| - |W_j^{r_0}| - b|W_i^{r_0}| \geq (2b + 2)\varepsilon m - \varepsilon m - 2b\varepsilon m \geq \varepsilon m. \]
	It follows from our assumption that $e_G(W_i^{r_0}, V_j\setminus \Gamma(W_i^{r_0})) > 0$, which is a contradiction.
	Hence in the end of this vertex-removal process we are left with $U_1 \subseteq V_1$ and $U_2 \subseteq V_2$ of sizes $|U_1| \geq (1-\varepsilon)|V_1|$ and $|U_2| \geq (1-\varepsilon)|V_2|$ such that the bipartite graph $G[U_1, U_2]$ is an $(\varepsilon m, b)$-bipartite-expander.
	
	We conclude by proving that $G[U_1, U_2]$ is in fact an $(ax, b)$-bipartite-expander.
	Assume, for contradiction, that for $1\leq i\neq j\leq 2$ there exists $W\subseteq U_i$ with $\varepsilon m < |W| \leq ax$ and such that $|\Gamma(W)\cap U_j| < b |W|$.
	Recall that $x = \min\{|V_1|, |V_2|\} \geq (2b + 2)\varepsilon m$, so it follows that
	\[ |U_j \setminus \Gamma(W)| > (1-\varepsilon)x - abx \geq (1 - \varepsilon - ab)x \geq (2b+2)(1 - \varepsilon - ab)\varepsilon m > \varepsilon m, \]
	and by the assumption we get $e(W, U_j\setminus \Gamma(W)) > 0$, which is a contradiction.
\end{proof}

\begin{corollary}\label{cor:expander}
	Let $G$ be a bipartite graph with parts $V_1, V_2$, let $m$ be some integer, let $0 < \varepsilon < \tfrac1{85}$, 
	and denote $x = \min(|V_1|, |V_2|)$.
	Assume that every two  subsets $V'_1, \subseteq V_1$, $V'_2\subseteq V_2$ with $|V'_1|, |V'_2| \geq \varepsilon m$ span at least one edge in $G$, i.e., $e(G[V'_1, V'_2]) > 0$.
	Then there exist $U_1, W_1\subseteq V_1$ and $U_2, W_2\subseteq V_2$ with $|U_1|, |W_1| \geq (1-\varepsilon)|V_1|$ and $|U_2|, |W_2| \geq (1-\varepsilon)|V_2|$ such that
	\begin{enumerate}
		\item If $|V_1|, |V_2| \geq \tfrac12m - 1$ then the bipartite graph $G[W_1, W_2]$ is a $(6\varepsilon x, \tfrac1{8\varepsilon}+1)$-bipartite-expander, and hence a $(12\varepsilon x, \tfrac1{16\varepsilon})$-expander.
		\item If $|V_1|, |V_2| \geq 20\varepsilon m$ then the bipartite graph $G[U_1, U_2]$ is a $(\tfrac1{10} x, 9)$-bipartite-expander, and hence a $(\tfrac15 x, 4)$-expander.
	\end{enumerate}
\end{corollary}

\subsection{Tree embeddings}\label{sec:tree}

We start by defining the following trees, playing a key role in our proofs.
\begin{definition}\label{def:Trhell}
	Let $T^{(r,h)}$ be the $r$-ary tree of depth $h$ (that is, the tree where each vertex, but a leaf, has $r$ children, and the distance, in edges, between the root and every leaf is exactly $h$).
	Let $T^{(r,h)}_\ell$ be the tree consisting of two disjoint copies of $T^{(r,h)}$ and a path of length $\ell$ connecting their roots.
\end{definition}

\begin{remark}\label{rmk:Trhell}
	Note that a longest path in $T^{(r,h)}_\ell$ is of length $\ell + 2h$.
	Furthermore, the tree $T^{(r,h)}_\ell$ has exactly $\ell - 1 + 2\cdot\frac{r^{h+1} - 1}{r - 1}$ vertices.
\end{remark}

The main ingredients in the proof of \Cref{lem:key} are the following claims regarding tree embeddings in bipartite-expander graphs.

\begin{proposition}\label{prop:TreeEmbd}
	Let $G$ be a bipartite graph with parts $V_1, V_2$ with $|V_1|, |V_2| \in \{\floor{m}, \ceil{m} \}$ for some positive number $m$.
	Let $0 < \varepsilon < \tfrac1{85}$ and assume that the pair $(V_1, V_2)$ satisfies the $\varepsilon$-property in $G$.
	Then $G$ contains every tree on at most $6 \varepsilon m$ vertices with maximum degree at most $\tfrac1{16\varepsilon} - 1$.
	In particular, $G$ contains a copy of $T^{(r,h)}_\ell$ where $r =\floor{ \tfrac1{16\varepsilon}} - 2$, $h = \ceil{\frac{\log (\varepsilon m)}{\log r}}$, and any  integer $\ell \in \left[1, 2\varepsilon m\right]$.
\end{proposition}

\begin{proof}
	By \Cref{cor:expander} there are subsets $U_1, \subseteq V_1$, $U_2\subseteq V_2$ for which the graph $G[U_1, U_2]$ is an $(12\varepsilon m, \tfrac1{16\varepsilon})$-expander.
	By \Cref{FP} we get that $G[U_1, U_2]$ contains a copy of any tree on at most $6\varepsilon m$ vertices with maximum degree at most $\tfrac1{16\varepsilon} - 1$.
	Set $r = \floor{\tfrac1{16\varepsilon}} - 2$, $h = \ceil{\frac{\log (\varepsilon m)}{\log r}}$, and $\ell \in [1, 2\varepsilon m]$.
	By \Cref{rmk:Trhell} the tree $T^{(r,h)}_\ell$ has at most $6\varepsilon m$ vertices and maximum degree at most $\tfrac1{16\varepsilon} - 1$, so in particular $G[U_1, U_2]$ contains a copy of it.
\end{proof}

\begin{proposition}\label{prop:LargeTreeEmbd}
	Let $G$ be a bipartite graph on parts $V_1, V_2$ with $|V_1|, |V_2| \in \{\floor{m}, \ceil{m} \}$ for some positive number $m$.
	Let $0 < \varepsilon < \tfrac1{85}$ and assume that the pair $(V_1, V_2)$ satisfies the $\varepsilon$-property in $G$.
	Then $G[V_1,V_2]$ contains a copy of $T^{(2,h)}_\ell$ for $h = \ceil{ \frac{\log (\varepsilon m)}{\log 2}}$, and any integer $\ell \in \left[1, 2(1 - 48\varepsilon)m\right]$.
	Moreover, if $\ell$ is even then we can embed a copy of $T^{(2,h)}_\ell$ with all leaves in $V_i$ for any $i\in\{1,2 \}$.
\end{proposition}

\begin{proof}
	Assume first that $\ell$ is odd.
	Let $U_{11}, U_{12} \subseteq V_1$ be disjoint, and $U_{21}, U_{22} \subseteq V_2$ be also disjoint, such that $|U_{ij}| = \ceil{21\varepsilon m}$ for any $i,j \in \{1,2\}$.
	By \Cref{cor:expander} (item $2$) applied separately on $G[U_{11}, U_{21}]$ and on $G[U_{12}, U_{22}]$ we get four subsets $W_{ij}\subseteq U_{ij}$, $i,j\in\{1,2\}$, all of size at least $20\varepsilon m$, such that each of the graphs $G[W_{11}, W_{21}]$ and $G[W_{12}, W_{22}]$ is a $(\tfrac12\varepsilon m, 4)$-expander.
	
	Let $X_1 \subseteq V_1 \setminus (W_{11}\cup W_{12})$ and let $X_2 \subseteq V_2\setminus (W_{21}\cup W_{22})$ be such that $|X_1| = |X_2| = \floor{(1 - 43\varepsilon)m}$
	Let $\ell \in [1, 2(1-48\varepsilon)m]$ be odd, and let $q=4\ceil{\varepsilon m}$.
	We now find a path of length exactly $\ell - 4 + q$.
	We do this using the following claim, implied by a standard DFS-based argument, stated implicitly in \cite{LongPathDFS} and more explicitly in, e.g., \cite{PokMonocPaths}.
	For a more extensive discussion about the DFS (Depth First Search) algorithm in finding paths in expander graphs we refer the reader to \cite{KrivDFS}.
	\begin{claim}\label{cl:LongPath}
		For every graph $G$ there exists a partition of its vertices $V = S \cup T \cup U$ such that $|S| = |T|$, $G$ has no edges between $S$ and $T$, and $U$ spans a path in $G$.
	\end{claim}
	Apply \Cref{cl:LongPath} to the graph $G[X_1, X_2]$.
	Notice that $|U| = |X_1 \cup X_2| - |S| - |T| = 2|X_1| - 2|S|$ and in particular $|U|$ is even.
	$U$ spans a path in $G[X_1, X_2]$, which is a bipartite graph, so we get $|U\cap X_1| = |U\cap X_2|$.
	Assume w.l.o.g.\ that $|S\cap X_1| \geq |S\cap X_2|$, then $|T\cap X_2| \geq |T\cap X_1|$.
	If $|S| = |T| \geq 2\ceil{\varepsilon m} - 1$ then we get $|S\cap X_1|, |T\cap X_2| \geq \varepsilon m$.
	However, we know that $e(S\cap X_1, T\cap X_2) \leq e(S, T) = 0$, contradicting the $\varepsilon$-property of the pair $(V_1, V_2)$ in $G$.
	Hence we get that $|S| = |T| \leq 2\ceil{\varepsilon m} - 2$, which means that $|U| \geq 2\floor{(1 - 43\varepsilon)m} - 4\ceil{\varepsilon m} + 4 \geq 2(1 - 45\varepsilon)m - 2$, and in particular $G[X_1, X_2]$ contains a path of length at least $2(1 - 45\varepsilon)m - 3$.
	Thus, let $P_0$ be a path of length $\ell - 4 + q \leq 2(1 - 45\varepsilon)m - 3$ and denote its endpoints by $u^*\in X_1$ and $v^*\in X_2$. 
	Let $u_1, \ldots, u_q$ be the first $q$ vertices of $P_0$ when moving from $u^*$, that is $u^* = u_1$, and let $v_1, \ldots v_q$ be the first $q$ vertices of $P_0$ when moving from $v^*$, that is $v^* = v_1$.
	Note that the vertices $\{u_1, \ldots, u_q \}$ are distributed equally between $X_1$ and $X_2$, having exactly $2\ceil{\varepsilon m}$ vertices in each set, and similarly the vertices $\{v_1, \ldots, v_q \}$.
	Consider now only the $2\ceil{\varepsilon m}$ vertices with odd indices, i.e., $\{u_1, u_3, \ldots, u_{q-1} \}$ and $\{v_1, v_3, \ldots, v_{q-1} \}$, and note that we have $\{u_1, u_3, \ldots, u_{q-1} \} \in X_1$ and $\{v_1, v_3, \ldots, v_{q-1} \} \in X_2$.
	Hence, by the $\varepsilon$-property of the pair $(V_1, V_2)$ in $G$, at least $\ceil{\varepsilon m} + 1$ of the vertices in $\{u_1, u_3, \ldots, u_{q-1} \}$ have some neighbor in $W_{21}$, and similarly, at least $\varepsilon m + 1$ of the vertices $\{v_1, v_3, \ldots, v_{q-1} \}$ have some neighbor in $W_{12}$.
	By the pigeonhole principle, there exists (an odd) $s\in \{1, \ldots, q-1 \}$ such that $u_s$ is connected to some vertex in $W_{21}$ and $v_{q-s}$ is connected to some vertex in $W_{12}$.
	Denote by $P$ the subpath of $P_0$ with endpoints $u_s$ and $v_{q-s}$, denoted by $u,v$, respectively, and note that it is of length exactly $\ell - 2$.

	Now, let $w_1$ be a neighbor of $u$ in $W_{21}$ and $w_2$ be a neighbor of $v$ in $W_{12}$.
	Recall that by \Cref{FP} there exists a copy of $T^{(2,h)}$ in $G[W_{11}, W_{21}]$, for $h = \ceil{\tfrac{\log (\varepsilon m)}{\log 2}}$, rooted in any predetermined vertex of $W_{21}$.
	Similarly, there exists a copy of $T^{2,h}$ in $G[W_{12}, W_{22}]$, for the same value of $h$, rooted in any predetermined vertex of $W_{12}$.
	Let $T_{w_1}, T_{w_2}$ be these copies of $T^{(2,h)}$ in $G[W_{11}, W_{21}]$ and in $G[W_{12}, W_{22}]$, respectively, rooted in $w_1\in W_{21}$ and in $w_2\in W_{12}$, respectively.
	Joining $T_{w_1}$ and $T_{w_2}$ to $P$, we get a copy of $T^{(2,h)}_\ell$, as required.
	
	If $\ell$ is even then we repeat the same argument, with a minor change.
	Note first that if $\ell$ is even then any embedded copy of $T^{(2,h)}_\ell$ in $G[V_1, V_2]$ has all leaves in either $V_1$ or $V_2$.
	Assume that we wish to embed a copy of $T^{(2,h)}_\ell$ with all leaves in $V_i$ for some $i\in \{1, 2\}$.
	Note further that if $\ell$ is even then $P_0$ is of an even length $\ell - 4 + q$, and hence both of its endpoints $u^*$ and $v^*$ are in $X_j$ for some $j\in \{1,2 \}$.
	Now, we look at $\{u_1, \ldots, u_q \}$ and $\{v_1, \ldots, v_q \}$ and split into two possible cases by the parity of $h$ and by the part in which the endpoints of $P_0$ are contained.
	If $h$ is even and $i\neq j$, or if $h$ is odd and $i=j$, then we consider only vertices of odd indices, i.e., $\{u_1, u_3, \ldots, u_{q-1} \}$ and $\{v_1, v_3, \ldots, v_{q-1} \}$.
	If $h$ is even and $i=j$, or if $h$ is odd and $i\neq j$, then we consider only vertices of even indices, i.e., $\{u_2, u_4, \ldots, u_q \}$ and $\{v_2, v_2, \ldots, v_q \}$.
	For simplicity we assume now that $h$ is even and $j=1$, $i=2$ (in particular $i\neq j$), where all other cases are handled similarly.
	This means that by the pigeon hole principle there exists (an odd) $s\in \{1, \ldots, q-1 \}$ such that $u_s$ is connected to some vertex in $W_{21}$ and $v_{q-s}$ is connected to some vertex in $W_{22}$, and equivalently to the odd $\ell$ case, we embed trees $T_{w_1}$ and $T_{w_2}$, having $w_1 \in W_{21}$ and $w_2 \in W_{22}$. 
\end{proof}

\subsection{Proof of the Key Lemma}\label{sec:PfKey}

We are now ready to prove \Cref{lem:key} using \Cref{prop:TreeEmbd} and \Cref{prop:LargeTreeEmbd}.

\begin{proof}[Proof of \Cref{lem:key}]
	Throughout the proof we denote $m \coloneqq \tfrac nk$.
	Note that $k$ is constant, so $m = \Theta(n)$.
	Recall that $S$ is the $\varepsilon$-graph obtained from $G_0$ with respect to the partition $\Pi=(V_1,\dots,V_k)$, that is, every edge $\{i,j\}\in E(S)$ represents a pair $(V_i,V_j)$ which satisfies the $\varepsilon$-property in $G_0$.
	
	The general idea is to convert a cycle (or a path) from the graph $S$ to a cycle in $G_0$ of the desired length, by using tree embeddings between clusters of $G_0$.
	Assume that $(1,\dots,b)$ is a cycle in $S$ and that $b$ is odd.
	Roughly speaking, we divide the cycle in $S$ into pairs of vertices that are connected with an edge $(2i,2i+1)$. We then embed in each pair of corresponding clusters $(V_{2i},V_{2i+1})$ a tree $T^{(r,h)}_\ell$ with appropriate parameters such that the leaf sets are in different clusters.
	Since each of these leaf sets contains at least $\varepsilon m$ vertices, we can use the $\varepsilon$-property to connect some leaf from the leaves in $V_{2i+i}$ and some leaf from the leaves in $V_{2i+2}$ by an edge.
	This way, we are able to connect different copies of $T^{(r,h)}_\ell$ to a very large tree, containing a copy of $T^{(r,h)}_{\ell^*}$ for an appropriate $\ell^*$, where its leaf sets are in $V_2$ and $V_b$. 
	We then use one vertex $v$ from $V_1$ and connect it to both leaf sets.
	This creates a cycle in $G_0$ of length exactly $t=\ell^*+2h+2$.
	For converting a path in $S$ to an even cycle in $G_0$ we use a similar argument, only this time we split each cluster into two clusters and use both endpoints of the path in $S$ to ``close" the cycle in $G_0$.
	We give the full details below.    

	We start with the first item.
	Suppose that $S$ contains a path of an odd length $b$, where $1 \leq b <k$, and let $t\in [\frac {C_1}{\log (1/\varepsilon)}\log n, (1-\delta)an]$ be even, $a\coloneqq \tfrac {b+1}k$.	
	Assume w.l.o.g.\ that this path is $(1, \ldots, b+1)$, and consider the sequence of corresponding clusters $V_1, \ldots, V_{b+1}$.
	We separate the case where $t$ is even into three parts.
	The first part deals with the case where $t \in [\frac {C_1}{\log (1/\varepsilon)}\log n, 2\varepsilon m]$, the second part deals with the case where $t\in [2\varepsilon m, (1-\delta)an]$ and $b=1$, and the third part deals with all other cases, i.e., $t\in [2\varepsilon m, (1-\delta)an]$ and $b\geq3$ (and is further separated into two subcases by the value of $b$).
	In each part we divide the vertices of the path into pairs, and embed a certain tree in the bipartite subgraph of the original graph induced by each pair.
	This is where we use the assumption of $b$ being odd, i.e., the path has an even number of vertices.
	A similar cluster pairing strategy was presented and used by Dellamonica et al.~\cite[Theorem 7]{ResilienceLongCyclesRandomGraph}.
	
	If $t \in [\frac {C_1}{\log (1/\varepsilon)}\log n, 2\varepsilon m]$ is even, then we look at a single edge in the path, say, $\{1, 2\}$.
	The graph $G_0[V_1, V_2]$ is bipartite and the pair $(V_1, V_2)$ satisfies the $\varepsilon$-property in $G_0$.
	By \Cref{prop:TreeEmbd} we know that $G_0[V_1, V_2]$ contains a copy of every tree with at most $6\varepsilon m$ vertices and maximum degree at most $\tfrac1{16\varepsilon} - 1$.
	In particular, $G_0[V_1, V_2]$ contains a copy of $T^{(r,h)}_\ell$ (as in \Cref{def:Trhell}) for $r = \floor{\tfrac 1{16\varepsilon}} - 2$, $h = \ceil{\tfrac{\log (\varepsilon m)}{\log r}}$ and any odd $\ell \in [1, 2\varepsilon m]$ (as $T^{(r,h)}_1$ has at most $4\varepsilon m$ vertices for these values of $r$ and $h$, and thus $T^{(r,h)}_\ell$ has at most $6\varepsilon m$).
	Note that a maximal path in $T^{(r,h)}_\ell$ is of length $2h + \ell$.
	Set $\ell = t - 2h - 1$ (note that it satisfies the constraints, as $1 \leq t - 2h - 1 \leq 2\varepsilon m$) and we get that a maximal path in a $T^{(r,h)}_\ell$-copy is of length exactly $t-1$.
	Now, note that this copy of $T^{(r,h)}_\ell$ has at least $\varepsilon m$ leaves in $V_1$ and $\varepsilon m$ leaves in $V_2$, due to parity considerations.
	By the $\varepsilon$-property of the pair $(V_1, V_2)$ in $G_0$ there is an edge between these two sets of leaves, closing a cycle of length $\ell+2h+1 = t$, as required.
	
	If $b=1$ and $t\in [2\varepsilon m, (1-\delta)an]$ is even, for $a\coloneqq \tfrac{b+1}k$, then once again the graph $G_0[V_1, V_2]$ is bipartite and the pair $(V_1, V_2)$ satisfies the $\varepsilon$-property in $G_0$.
	We repeat the previous argument but with the only change of embedding a different tree in $G_0[V_1, V_2]$.
	By \Cref{prop:LargeTreeEmbd} we know that $G_0[V_1, V_2]$ contains a copy of $T^{(2,h)}_\ell$  for  $h = \ceil{\tfrac{\log (\varepsilon m)}{\log 2}}$ and $\ell = t-1-2h$.
	Also here, note that this copy of $T^{(2,h)}_\ell$ has at least $\varepsilon m$ leaves in $V_1$ and $\varepsilon m$ leaves in $V_2$, due to parity considerations.
	Again, by the $\varepsilon$-property of the pair $(V_1, V_2)$ in $G_0$ there is an edge between these two sets of leaves, closing a cycle of length $t$, as required.

	If $b\geq 3$ and $t \in [2\varepsilon m, (1-\delta)an]$ is even, for $a\coloneqq \tfrac{b+1}k$, then we look at the full path $(1, \ldots, b+1)$ and the set of corresponding clusters $V_1, \ldots, V_{b+1}$.
	Informally, we embed two copies of $T^{(2, h)}_\ell$ for some carefully chosen values $h, \ell$, one in $G_0[V_1, V_2]$, and one in $G_0[V_b, V_{b+1}]$.
	Then, if we have used all the clusters already for tree embedding (i.e., $b=3$), then we connect these two trees by two edges to create a cycle of the desired length.
	Otherwise, we keep embedding trees in all clusters we have not touched yet.
	Formally, we further separate this case into two subcases and argue as follows.
	
	Assume first that $b = 3$.
	For following the arguments of this subcase \Cref{fig:KeyLemmab=3} can be helpful.
	Note that each of the pairs $(V_1, V_2)$ and $(V_3, V_4)$ satisfies the $\varepsilon$-property in $G_0$, and that we have $|V_j| \in \{\floor{m}, \ceil{m} \}$ for any $j \in [4]$.
	Now let $j\in \{1, 3 \}$.
	By \Cref{prop:LargeTreeEmbd} we know that $G_0[V_j, V_{j+1}]$ contains a copy of $T^{2,h}_{\ell_j}$ where $h = \ceil{\frac{\log (\varepsilon m)}{\log 2}}$ and $\ell_j$ is such that $\ell_1 + \ell_3 = t - 4h - 2$, $|\ell_1 - \ell_3| \leq 2$, and both are even.
	Note here that $\ell_j \leq \tfrac12 t - 2h \leq 2(1 - 48\varepsilon)m$.
	We embed two such $T^{(2,h)}_{\ell_j}$-copies, $j\in \{1, 3 \}$, such that the leaf sets $L_2, L'_2$ and $L_3, L'_3$ are in $V_2$ and $V_3$, respectively (which is possible as $\ell_j$ is even).
	Having $|L_2|, |L'_2|, |L_3|, |L'_3| \geq \varepsilon m$, by the $\varepsilon$-property of the pair $(V_2, V_3)$ in $G_0$, there exist two edges, one between $L_2$ and $L_3$, and the other between $L'_2$ and $L'_3$.
	These two edges close a cycle of length exactly $t$.

\iffigure
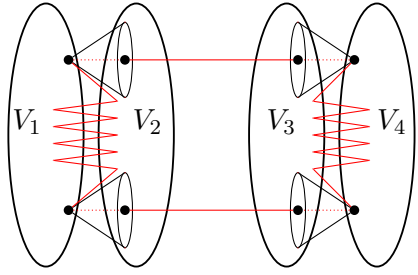
\begin{figure}
	\centering
	\begin{tikzpicture}[auto, vertex/.style={circle,draw=black!100,fill=black!100, thick,
		inner sep=0pt,minimum size=1mm}]
	\tikzstyle{novertex}=[circle, draw, red, inner sep=0pt, minimum size=0.1pt]

	\draw [line width=0.25mm] (-2.2,0.75) ellipse (0.5cm and 1.75cm);
	\node [label=$V_1$] at (-2.45,0.5) {};
	\node (1bot) at (-1.9,-0.25) [vertex] {};
	\node (1top) at (-1.9,1.75) [vertex] {};

	\draw [line width=0.25mm] (-1,0.75) ellipse (0.5cm and 1.75cm);
	\node [label={[shift={(0.4,0)}]$V_2$}] at (-1.25,0.5) {};
	\draw (-1.15,-0.25) ellipse (0.1cm and 0.5cm);
	\node (2botellipse) at (-1.15,-0.25) [vertex] {};
	\node (2botup) at (-1.15,0.25) [novertex] {};
	\node (2botdown) at (-1.15,-0.75) [novertex] {};
	\draw [-] (1bot) --node[inner sep=0pt,swap]{} (2botup);		
	\draw [-] (1bot) --node[inner sep=0pt,swap]{} (2botdown);
	\draw (-1.15,1.75) ellipse (0.1cm and 0.5cm);
	\node (2topellipse) at (-1.15,1.75) [vertex] {};
	\node (2topup) at (-1.15,2.25) [novertex] {};
	\node (2topdown) at (-1.15,1.25) [novertex] {};
	\draw [-] (1top) --node[inner sep=0pt,swap]{} (2topup);		
	\draw [-] (1top) --node[inner sep=0pt,swap]{} (2topdown);
	\node (2botpath) at (-1.25,0.3) [novertex] {};
	\node (1botpath) at (-2.1,0.4125) [novertex] {};
	\node (2botmidpath) at (-1.25,0.525) [novertex] {};
	\node (1botmidpath) at (-2.1,0.6375) [novertex] {};
	\node (2midpath) at (-1.25,0.75) [novertex] {};
	\node (1topmidpath) at (-2.1,0.8625) [novertex] {};
	\node (2topmidpath) at (-1.25,0.975) [novertex] {};
	\node (1toppath) at (-2.1,1.0875) [novertex] {};
	\node (2toppath) at (-1.25,1.2) [novertex] {};
	\draw [red,-] (1top) --node[inner sep=0pt,swap]{} (2toppath);
	\draw [red,-] (2toppath) --node[inner sep=0pt,swap]{} (1toppath);
	\draw [red,-] (1toppath) --node[inner sep=0pt,swap]{} (2topmidpath);
	\draw [red,-] (2topmidpath) --node[inner sep=0pt,swap]{} (1topmidpath);
	\draw [red,-] (1topmidpath) --node[inner sep=0pt,swap]{} (2midpath);
	\draw [red,-] (2midpath) --node[inner sep=0pt,swap]{} (1botmidpath);
	\draw [red,-] (1botmidpath) --node[inner sep=0pt,swap]{} (2botmidpath);
	\draw [red,-] (2botmidpath) --node[inner sep=0pt,swap]{} (1botpath);
	\draw [red,-] (1botpath) --node[inner sep=0pt,swap]{} (2botpath);
	\draw [red,-] (2botpath) --node[inner sep=0pt,swap]{} (1bot);

	\draw [line width=0.25mm] (2.2,0.75) ellipse (0.5cm and 1.75cm);
	\node [label=$V_4$] at (2.4,0.5) {};
	\node (4bot) at (1.9,-0.25) [vertex] {};
	\node (4top) at (1.9,1.75) [vertex] {};

	\draw [line width=0.25mm] (1,0.75) ellipse (0.5cm and 1.75cm);
	\node [label={[shift={(-0.07,0)}]$V_3$}] at (1,0.5) {};
	\draw (1.15,-0.25) ellipse (0.1cm and 0.5cm);
	\node (3botellipse) at (1.15,-0.25) [vertex] {};
	\node (3botup) at (1.15,0.25) [novertex] {};
	\node (3botdown) at (1.15,-0.75) [novertex] {};
	\draw [-] (4bot) --node[inner sep=0pt,swap]{} (3botup);		
	\draw [-] (4bot) --node[inner sep=0pt,swap]{} (3botdown);
	\draw (1.15,1.75) ellipse (0.1cm and 0.5cm);
	\node (3topellipse) at (1.15,1.75) [vertex] {};
	\node (3topup) at (1.15,2.25) [novertex] {};
	\node (3topdown) at (1.15,1.25) [novertex] {};
	\draw [-] (4top) --node[inner sep=0pt,swap]{} (3topup);
	\draw [-] (4top) --node[inner sep=0pt,swap]{} (3topdown);
	\node (3botpath) at (1.35,0.3) [novertex] {};
	\node (4botpath) at (2.1,0.4125) [novertex] {};
	\node (3botmidpath) at (1.35,0.525) [novertex] {};
	\node (4botmidpath) at (2.1,0.6375) [novertex] {};
	\node (3midpath) at (1.35,0.75) [novertex] {};
	\node (4topmidpath) at (2.1,0.8625) [novertex] {};
	\node (3topmidpath) at (1.35,0.975) [novertex] {};
	\node (4toppath) at (2.1,1.0875) [novertex] {};
	\node (3toppath) at (1.35,1.2) [novertex] {};
	\draw [red,-] (4top) --node[inner sep=0pt,swap]{} (3toppath);
	\draw [red,-] (3toppath) --node[inner sep=0pt,swap]{} (4toppath);
	\draw [red,-] (4toppath) --node[inner sep=0pt,swap]{} (3topmidpath);
	\draw [red,-] (3topmidpath) --node[inner sep=0pt,swap]{} (4topmidpath);
	\draw [red,-] (4topmidpath) --node[inner sep=0pt,swap]{} (3midpath);
	\draw [red,-] (3midpath) --node[inner sep=0pt,swap]{} (4botmidpath);
	\draw [red,-] (4botmidpath) --node[inner sep=0pt,swap]{} (3botmidpath);
	\draw [red,-] (3botmidpath) --node[inner sep=0pt,swap]{} (4botpath);
	\draw [red,-] (4botpath) --node[inner sep=0pt,swap]{} (3botpath);
	\draw [red,-] (3botpath) --node[inner sep=0pt,swap]{} (4bot);

	\draw [red,densely dotted] (1top) --node[inner sep=0pt,swap]{} (2topellipse);
	\draw [red,-] (2topellipse) --node[inner sep=0pt,swap]{} (3topellipse);
	\draw [red,densely dotted] (4top) --node[inner sep=0pt,swap]{} (3topellipse);
	\draw [red,densely dotted] (1bot) --node[inner sep=0pt,swap]{} (2botellipse);
	\draw [red,-] (2botellipse) --node[inner sep=0pt,swap]{} (3botellipse);
	\draw [red,densely dotted] (4bot) --node[inner sep=0pt,swap]{} (3botellipse);

	\end{tikzpicture}
	\caption{Embedding trees to create an even cycle (in red), proof of \Cref{lem:key}, $b=3$. (For the simplicity of the figure the roots of the $T^{(r,h)}_\ell$-copies are contained in $V_1$ and $V_4$, but they can rather be contained in $V_2$ and $V_3$, respectively, as well).}\label{fig:KeyLemmab=3}
\end{figure}
\fi

	Assume now that $b\geq 5$.
	When following the arguments of this subcase \Cref{fig:KeyLemma} can be helpful.
	In this subcase too we embed two $T^{(2,h)}_\ell$-copies, for a suitable choice of $h,\ell$, in $G_0[V_1, V_2]$ and in $G_0[V_b, V_{b+1}]$.
	However, we do not connect them directly by two edges, but through other $T^{(2,h)}_\ell$-copies we embed in the rest of clusters.
	More precisely, for each $j\in \{3,\ldots, b-1 \}$, arbitrarily split the vertex set $V_j$ into two equally sized subsets (up to possibly one vertex), denoted by $U_j,U'_j$.
	Set some $i \in \{2, \tfrac12 (b - 1) \}$ and look at the pair $(V_{2i-1}, V_{2i})$.
	Since $(V_{2i-1}, V_{2i})$ satisfies the $\varepsilon$-property in $G_0$, it follows that each of the pairs $(U_{2i-1}, U_{2i})$ and $(U'_{2i-1}, U'_{2i})$ satisfies the $\varepsilon'$-property in $G_0$, for $\varepsilon'$ satisfying $\varepsilon'(\tfrac12 m - 1) = \varepsilon m$ (namely, every two subsets, one from each set of the pair, of size at least $\varepsilon' (\tfrac12 m - 1)$ each, span an edge in $G_0$).
	We have $|U_{2i-1}|, |U_{2i}| \geq \tfrac12 m-1$, so by \Cref{prop:LargeTreeEmbd} (taking $\tfrac 12m-1$ instead of $m$) we get that $G_0[U_{2i-1}, U_{2i}]$ contains a copy of $T^{(2,h)}_{\ell_0}$ for $h = \ceil{\frac{\log (\varepsilon m)}{\log 2}}$ and $\ell_0 = \lfloor\tfrac{t - b +  1}{b+1}\rfloor_{odd} - 2h \leq 2(1 - 48\varepsilon)(\frac 12m-1)$, where $\lfloor x \rfloor_{odd}$ is the odd integer $y$ such that $y\leq x$ and $x-y< 2$.
	Denote this copy by $T_{2i-1, 2i}$ and its leaf sets in $U_{2i-1}, U_{2i}$ by $L_{2i-1}, L_{2i}$, respectively.
	We do the same for $G_0[U'_{2i-1}, U'_{2i}]$ where we denote the embedded copy of $T^{(2,h)}_{\ell_0}$ by $T'_{2i-1, 2i}$, and its leaf sets in $U'_{2i-1}, U'_{2i}$ by $L'_{2i-1}, L'_{2i}$, respectively.
	Do this for every $i\in \{2, \tfrac12 (b - 1) \}$ with the same notations.   
	If $b \geq 7$, then recall that for every $i \in \{2, \tfrac12 (b - 3) \}$ each of the pairs $(U_{2i}, U_{2i+1})$ and $(U'_{2i}, U'_{2i+1})$ satisfies the $\varepsilon'$-property in $G_0$, and moreover, note that we have $|L_{2i}|, |L_{2i+1}|, |L'_{2i}|, |L'_{2i+1}| \geq \varepsilon m$.
	Thus, for every $i \in \{2, \tfrac12 (b - 3) \}$ we have $e_{G_0}(L_{2i}, L_{2i+1}), e_{G_0}(L'_{2i}, L'_{2i+1}) > 0$, so we add an edge between every such two leaf sets, summing up to total of $b-5$ new edges.
	If $b = 5$ then there is only one pair of clusters we have splitted, $(V_3, V_4)$, so we not yet add any edges.
	This creates two disjoint copies of $T^{(2,h)}_{\ell^*}$ in $G_0$, where $h=\ceil{\frac{\log (\varepsilon m)}{\log 2}}$, $\ell^*=\frac 12(\ell_0+2h)(b-3)+\frac 12(b-5)-2h$, one contained in $U \coloneqq \bigcup_{j=3}^{b-1} U_j$ and the other in $U'\coloneqq \bigcup_{j=3}^{b-1} U'_j$.
	Moreover, the first $T^{(2,h)}_{\ell^*}$-copy, embedded in $U$, has at least $\varepsilon m$ leaves in $U_3$ and at least $\varepsilon m$ leaves in $U_{b-1}$.
	Similarly, the other $T^{(2, h)}_{\ell^*}$-copy, embedded in $U'$, has at least $\varepsilon m$ leaves in $U'_3$ and at least $\varepsilon m$ leaves in $U'_{b-1}$ (see \Cref{fig:KeyLemma}).
	Now, we treat the pairs $(V_j, V_{j+1})$ where $j \in \{1, b \}$ almost similarly to how we treated them in the subcase $b=3$.
	More formally, we note that $(V_j,V_{j+1})$ also has the $\varepsilon$-property in $G_0$ and that $|V_j|,|V_{j+1}| \in \{\floor{m}, \ceil{m} \}$.
	So by \Cref{prop:LargeTreeEmbd} we get that $G_0[V_j,V_{j+1}]$ contains a copy of $T^{(2,h)}_{\ell_j}$ for $h = \ceil{\frac{\log (\varepsilon m)}{\log 2}}$ and the $\ell_j$'s are such that $\ell_1+\ell_b = (t-2\ell^*-4h)-4h-4$, $|\ell_{1}-\ell_{b}|\leq 2$, and both are even (which means that $\ell_j \leq \frac 12 (t-2\ell^*-4h)-2h-1 \leq 2(1 - 48\varepsilon)m$), where the leaf sets $L_2,L_2'$, and $L_b,L_b'$ are in $V_2$ and $V_b$, respectively (as $\ell_j$ is even).
	Since $|L_{2}|, |L_2'|, |L_b|, |L_b'| \geq \varepsilon m$, once again, by the $\varepsilon$-property, we can connect some $v_2\in L_{2}$ with  $v_3\in L_3$, some $v_2'\in L_{2}'$ with $v_3'\in L_3'$, some $v_{b-1}\in L_{b-1}$ with $v_b\in L_b$, and some $v_{b-1}'\in L_{b-1}'$ with $v_b'\in L_b'$ (see \Cref{fig:KeyLemma}).
	By doing that we complete a cycle of length exactly $\ell_1 + \ell_b + 2\ell^* + 8h + 4 = t$.

	We now prove the second item.
	Suppose now that $S$ contains an odd cycle of length $b$, where $3\leq b < k$, and let $t \in \left[\tfrac{(b-1)\cdot C_1}{2\log(1/\varepsilon)}\log n, (1-\delta)an \right]$ be odd, $a = \tfrac bk$.
	Assume w.l.o.g.\ that this cycle is $(1, \ldots, b)$ and consider the set of corresponding clusters $V_1, \ldots, V_b$.
	Let $i\in \{1, \ldots,  \tfrac12 (b-1) \}$ and look at the pair $(V_{2i-1}, V_{2i})$.
	Using \Cref{prop:TreeEmbd} and \Cref{prop:LargeTreeEmbd} we embed one of two different possible trees in $G_0[V_{2i-1}, V_{2i}]$, depending on the value of $t$, to eventually create a cycle of the required length.
	Recall that the pair $(V_{2i-1}, V_{2i})$ satisfies the $\varepsilon$-property in $G_0$, and furthermore, that $|V_{2i-1}|, |V_{2i}| \geq \floor{m}$.
	Hence, by \Cref{prop:TreeEmbd} and \Cref{prop:LargeTreeEmbd}, $G_0[V_{2i-1}, V_{2i}]$ contains a copy of $T^{(r,h)}_\ell$ where $h = \ceil{\tfrac{\log (\varepsilon m)}{\log r}}$ for both $r = \floor{\tfrac1{16\varepsilon}} - 2, \ell\in[1, 2\varepsilon m]$ and $r = 2, \ell \in [1, 2(1-48\varepsilon)m]$, respectively.
	Thus, we embed a copy of $T^{(r,h)}_{\ell_i}$ in $G_0[V_{2i-1}, V_{2i}]$ for $h = \ceil{\tfrac{\log (\varepsilon m)}{\log r}}$, where $r = \floor{\tfrac1{16\varepsilon}} - 2$ if $t\in \left[\tfrac{(b-1)\cdot C_1}{2\log(1/\varepsilon)}\log n, 2\varepsilon m\right]$, and $r = 2$ if $t\in \left[2\varepsilon m, (1-\delta)an\right]$.
	We choose the value of $\ell_i$ as follows.
	For all $i\in \{2, \ldots, \tfrac12(b-1) \}$ we set $\ell_i = \ell_0 \coloneqq \floor{\tfrac{2t - 2 - \left(1 + 2h \right)(b-1)}{b-1}}_{odd}$, and $\ell_1 = t - 1 - \tfrac{b-1}2 - h(b-1) - \tfrac12(b-3)\ell_0$.
	Note that $\ell_1$ is also odd, and moreover, that $\ell_0, \ell_1 \in [1, 2(1-48\varepsilon)m]$.
	For every $i\in \{1, \ldots, \tfrac12(b-1) \}$ we denote the embedded $T^{(r,h)}_{\ell_i}$-copy in $G_0[V_{2i-1}, V_{2i}]$ by $T_{2i-1, 2i}$, and further denote by $L_{2i-1}, L_{2i}$ its leaf sets in $V_{2i-1}$ and in $V_{2i}$, respectively.
	Note that for every $i\in \{1, \ldots, \tfrac12(b-1)\}$, a maximal path in $T_{2i-1, 2i}$ is of length $2h + \ell_i$.
	Recall that for every $i\in \{1, \ldots,  \tfrac12(b-3) \}$, also the pair $(V_{2i}, V_{2i+1})$ satisfies the $\varepsilon$-property in $G_0$, and note that we have $|L_{2i}|, |L_{2i+1}| \geq \varepsilon m$.
	Thus we have $e_{G_0}(L_{2i}, L_{2i+1}) > 0$ for every $i\in \{1, \ldots, \tfrac12(b-3) \}$, so we add an edge between every such pair of leaf sets, summing up to $\tfrac12(b-3)$ new edges.
	Thus we get in $G_0$ a copy of the tree $T^{(r,h)}_{\ell^*}$, where $\ell^* = \sum_{i=1}^{\tfrac12(b-1)}\ell_i + (b-3)h + \tfrac12(b-3) = t - 2h - 2$, with at least $\varepsilon m$ leaves in $V_1$ and at least $\varepsilon m$ leaves in $V_{b-1}$.
	Some maximal path inside this tree (connecting the mentioned two leaf sets) will be used to get a cycle of length $t$ along with extra two edges.
	Now, we note that there exists a vertex $v_b\in V_b$ which is adjacent both to a vertex in $L_1$ and a vertex in $L_{b-1}$.
	Indeed, otherwise one of $L_1, L_{b-1}$ would have fewer than $(1-\varepsilon)\floor{m}$ neighbors in $V_b$, which contradicts the $\varepsilon$-property of the pairs $(V_1, V_b)$ and $(V_{b-1}, V_b)$ in $G_0$.
	Thus we can connect the vertex $v_b$ to a vertex in $L_1$ and to a vertex in $L_{b-1}$, adding two more edges and closing a cycle of length exactly $t$.
\end{proof}

\iffigure
\begin{figure}
	\centering
	\begin{tikzpicture}[auto, vertex/.style={circle,draw=black!100,fill=black!100, thick,
		inner sep=0pt,minimum size=1mm}]
		\tikzstyle{novertex}=[circle, draw, inner sep=0pt, minimum size=0.1pt]
		\tikzstyle{rednovertex}=[circle, draw, red, inner sep=0pt, minimum size=0.1pt]

		\draw[line width=0.25mm][rounded corners=20pt] (0,1) rectangle ++(6.5,1.5) {};
		\node [label=$U$] at (3.25,1.75) {};
		\node (uleft) at (1,1.75) [vertex] {};
		\draw (0.5,1.75) ellipse (0.1cm and 0.5cm);
		\node (uleftellipse) at (0.5,1.75) [vertex] {};
		\node (uleftup) at (0.5,2.25) [novertex] {};
		\node (uleftdown) at (0.5,1.25) [novertex] {};		
		\draw [-] (uleft) --node[inner sep=0pt,swap]{} (uleftup);		
		\draw [-] (uleft) --node[inner sep=0pt,swap]{} (uleftdown);
		\draw [red, densely dotted] (uleft) --node[inner sep=0pt,swap]{} (uleftellipse);
		\node (uright) at (5.5,1.75) [vertex] {};
		\draw (6,1.75) ellipse (0.1cm and 0.5cm);
		\node (urightellipse) at (6,1.75) [vertex] {};
		\node (urightup) at (6,2.25) [novertex] {};
		\node (urightdown) at (6,1.25) [novertex] {};
		\draw [-] (uright) --node[inner sep=0pt,swap]{} (urightup);		
		\draw [-] (uright) --node[inner sep=0pt,swap]{} (urightdown);
		\draw [red, densely dotted] (uright) --node[inner sep=0pt,swap]{} (urightellipse);
		\draw [red, densely dotted] (uright) --node[inner sep=0pt,swap]{} (uleft);

		\draw[line width=0.25mm][rounded corners=20pt] (0,-1) rectangle ++(6.5,1.5) {};
		\node [label=$U'$] at (3.25,-0.25) {};
		\node (u'left) at (1,-0.25) [vertex] {};
		\draw (0.5,-0.25) ellipse (0.1cm and 0.5cm);
		\node (u'leftellipse) at (0.5,-0.25) [vertex] {};		
		\node (u'leftup) at (0.5,0.25) [novertex] {};
		\node (u'leftdown) at (0.5,-0.75) [novertex] {};
		\draw [-] (u'left) --node[inner sep=0pt,swap]{} (u'leftup);		
		\draw [-] (u'left) --node[inner sep=0pt,swap]{} (u'leftdown);
		\draw [red, densely dotted] (u'left) --node[inner sep=0pt,swap]{} (u'leftellipse);
		\node (u'right) at (5.5,-0.25) [vertex] {};
		\draw (6,-0.25) ellipse (0.1cm and 0.5cm);
		\node (u'rightellipse) at (6,-0.25) [vertex] {};
		\node (u'rightup) at (6,0.25) [novertex] {};
		\node (u'rightdown) at (6,-0.75) [novertex] {};
		\draw [-] (u'right) --node[inner sep=0pt,swap]{} (u'rightup);		
		\draw [-] (u'right) --node[inner sep=0pt,swap]{} (u'rightdown);
		\draw [red, densely dotted] (u'right) --node[inner sep=0pt,swap]{} (u'rightellipse);
		\draw [red, densely dotted] (u'right) --node[inner sep=0pt,swap]{} (u'left);

		\draw [line width=0.25mm] (-2.2,0.75) ellipse (0.5cm and 1.75cm);
		\node [label=$V_1$] at (-2.45,0.5) {};
		\node (1bot) at (-1.9,-0.25) [vertex] {};
		\node (1top) at (-1.9,1.75) [vertex] {};

		\draw [line width=0.25mm] (-1,0.75) ellipse (0.5cm and 1.75cm);
		\node [label={[shift={(0.4,0)}]$V_2$}] at (-1.25,0.5) {};
		\draw (-1.15,-0.25) ellipse (0.1cm and 0.5cm);
		\node (2botellipse) at (-1.15,-0.25) [vertex] {};
		\node (2botup) at (-1.15,0.25) [novertex] {};
		\node (2botdown) at (-1.15,-0.75) [novertex] {};
		\draw [-] (1bot) --node[inner sep=0pt,swap]{} (2botup);		
		\draw [-] (1bot) --node[inner sep=0pt,swap]{} (2botdown);
		\draw [red, densely dotted] (1bot) --node[inner sep=0pt,swap]{} (2botellipse);
		\draw (-1.15,1.75) ellipse (0.1cm and 0.5cm);
		\node (2topellipse) at (-1.15,1.75) [vertex] {};
		\node (2topup) at (-1.15,2.25) [novertex] {};
		\node (2topdown) at (-1.15,1.25) [novertex] {};
		\draw [-] (1top) --node[inner sep=0pt,swap]{} (2topup);		
		\draw [-] (1top) --node[inner sep=0pt,swap]{} (2topdown);
		\draw [red, densely dotted] (1top) --node[inner sep=0pt,swap]{} (2topellipse);
		\draw [red,-] (uleftellipse) --node[inner sep=0pt,swap]{} (2topellipse);
		\draw [red,-] (u'leftellipse) --node[inner sep=0pt,swap]{} (2botellipse);
		\node (2botpath) at (-1.25,0.3) [rednovertex] {};
		\node (1botpath) at (-2.1,0.4125) [rednovertex] {};
		\node (2botmidpath) at (-1.25,0.525) [rednovertex] {};
		\node (1botmidpath) at (-2.1,0.6375) [rednovertex] {};
		\node (2midpath) at (-1.25,0.75) [rednovertex] {};
		\node (1topmidpath) at (-2.1,0.8625) [rednovertex] {};
		\node (2topmidpath) at (-1.25,0.975) [rednovertex] {};
		\node (1toppath) at (-2.1,1.0875) [rednovertex] {};
		\node (2toppath) at (-1.25,1.2) [rednovertex] {};
		\draw [red,-] (1top) --node[inner sep=0pt,swap]{} (2toppath);
		\draw [red,-] (2toppath) --node[inner sep=0pt,swap]{} (1toppath);
		\draw [red,-] (1toppath) --node[inner sep=0pt,swap]{} (2topmidpath);
		\draw [red,-] (2topmidpath) --node[inner sep=0pt,swap]{} (1topmidpath);
		\draw [red,-] (1topmidpath) --node[inner sep=0pt,swap]{} (2midpath);
		\draw [red,-] (2midpath) --node[inner sep=0pt,swap]{} (1botmidpath);
		\draw [red,-] (1botmidpath) --node[inner sep=0pt,swap]{} (2botmidpath);
		\draw [red,-] (2botmidpath) --node[inner sep=0pt,swap]{} (1botpath);
		\draw [red,-] (1botpath) --node[inner sep=0pt,swap]{} (2botpath);
		\draw [red,-] (2botpath) --node[inner sep=0pt,swap]{} (1bot);

		\draw [line width=0.25mm] (8.7,0.75) ellipse (0.5cm and 1.75cm);
		\node [label=$V_b$] at (8.9,0.5) {};
		\node (bbot) at (8.4,-0.25) [vertex] {};
		\node (btop) at (8.4,1.75) [vertex] {};

		\draw [line width=0.25mm] (7.5,0.75) ellipse (0.5cm and 1.75cm);
		\node [label={[shift={(-0.07,0)}]$V_{b-1}$}] at (7.5,0.5) {};
		\draw (7.65,-0.25) ellipse (0.1cm and 0.5cm);
		\node (b1botellipse) at (7.65,-0.25) [vertex] {};
		\node (b1botup) at (7.65,0.25) [novertex] {};
		\node (b1botdown) at (7.65,-0.75) [novertex] {};
		\draw [-] (bbot) --node[inner sep=0pt,swap]{} (b1botup);		
		\draw [-] (bbot) --node[inner sep=0pt,swap]{} (b1botdown);
		\draw [red, densely dotted] (bbot) --node[inner sep=0pt,swap]{} (b1botellipse);
		\draw (7.65,1.75) ellipse (0.1cm and 0.5cm);
		\node (b1topellipse) at (7.65,1.75) [vertex] {};
		\node (b1topup) at (7.65,2.25) [novertex] {};
		\node (b1topdown) at (7.65,1.25) [novertex] {};
		\draw [-] (btop) --node[inner sep=0pt,swap]{} (b1topup);
		\draw [-] (btop) --node[inner sep=0pt,swap]{} (b1topdown);
		\draw [red, densely dotted] (btop) --node[inner sep=0pt,swap]{} (b1topellipse);
		\draw [red,-] (urightellipse) --node[inner sep=0pt,swap]{} (b1topellipse);
		\draw [red,-] (u'rightellipse) --node[inner sep=0pt,swap]{} (b1botellipse);
		\node (b1botpath) at (7.85,0.3) [rednovertex] {};
		\node (bbotpath) at (8.6,0.4125) [rednovertex] {};
		\node (b1botmidpath) at (7.85,0.525) [rednovertex] {};
		\node (bbotmidpath) at (8.6,0.6125) [rednovertex] {};
		\node (b1midpath) at (7.85,0.75) [rednovertex] {};
		\node (btopmidpath) at (8.6,0.8625) [rednovertex] {};
		\node (b1topmidpath) at (7.85,0.975) [rednovertex] {};
		\node (btoppath) at (8.6,1.0875) [rednovertex] {};
		\node (b1toppath) at (7.85,1.2) [rednovertex] {};
		\draw [red,-] (btop) --node[inner sep=0pt,swap]{} (b1toppath);
		\draw [red,-] (b1toppath) --node[inner sep=0pt,swap]{} (btoppath);
		\draw [red,-] (btoppath) --node[inner sep=0pt,swap]{} (b1topmidpath);
		\draw [red,-] (b1topmidpath) --node[inner sep=0pt,swap]{} (btopmidpath);
		\draw [red,-] (btopmidpath) --node[inner sep=0pt,swap]{} (b1midpath);
		\draw [red,-] (b1midpath) --node[inner sep=0pt,swap]{} (bbotmidpath);
		\draw [red,-] (bbotmidpath) --node[inner sep=0pt,swap]{} (b1botmidpath);
		\draw [red,-] (b1botmidpath) --node[inner sep=0pt,swap]{} (bbotpath);
		\draw [red,-] (bbotpath) --node[inner sep=0pt,swap]{} (b1botpath);
		\draw [red,-] (b1botpath) --node[inner sep=0pt,swap]{} (bbot);

	\end{tikzpicture}
	\caption{Embedding trees to create an even cycle (in red), proof of \Cref{lem:key}, $b\ge 5$. (For the simplicity of the figure the roots of the $T^{(r,h)}_\ell$-copies are contained in $V_1$ and $V_b$, but they can rather be contained in $V_2$ and $V_{b-1}$, respectively, as well).}\label{fig:KeyLemma}
\end{figure}
\fi

\section{Robustness}\label{sec:robust}

In this section we prove \Cref{thm:mainRobustness}, and discuss its tightness.
We show that this result is tight for many values of $t$, and we prove \Cref{cl:TightOdd}, giving a tighter result for the cases in which \Cref{thm:mainRobustness} is not tight enough.

For the proofs in this section we use Szemer\'edi's celebrated Regularity Lemma~\cite{RegularityLemma}. 

\begin{theorem}[Szemer\'edi's Regularity Lemma \cite{RegularityLemma}]\label{thm:RegularityLemma}
	For every positive real $\varepsilon$ and for every positive integer $k_0$ there are positive integers $n_0$ and $K_0$ with the following property: for every graph $G$ on $n\geq n_0$ vertices there is an $\varepsilon$-regular partition $\Pi=(V_1,\dots,V_k)$ of $V(G)$ such that $\left||V_i|-|V_j|\right|\leq  1$ and $k_0\leq k\leq K_0$.
\end{theorem}

The following lemma bounds from below the number of edges in the reduced graph $R$ of the graph $G$ from \Cref{thm:mainRobustness}, similarly to \Cref{lem:EdgesRGraph}.

\begin{lemma}\label{lem:EdgeRGraphRbstns}
	Let $\beta>0$ and $\varepsilon\leq \frac {\beta}{100}$.
	Let $G$ be a graph on $n\geq n_0$ vertices with an $\varepsilon$-regular partition $\Pi=(V_1,\dots,V_k)$ provided by the Regularity Lemma with parameters $\varepsilon$ and $k\geq \frac {5}{\beta}$.
	Assume that $e(G)\geq (x+\beta) \binom n2$ for a constant $0\leq x<1-\beta$.
	Let $R:=R(G,\Pi, \rho, \varepsilon)$ be the reduced graph as in \Cref{def:RGraph} (and as mentioned in \Cref{def:RGraph}, here $p=1$) where $\rho=10\varepsilon$.
	Then $e(R)\geq(x+\beta/2) \binom k2$.
\end{lemma}

\begin{proof}
	Let $G'$ be the subgraph of $G$ obtained by keeping only the edges between the clusters $V_i,V_j$ for which $\{i,j\}\in E(R)$.
	We count the edges of $G - G'$ as follows.
	\begin{itemize}
		\item Edges in non-regular pairs.
		There are at most $\varepsilon \binom k2 \tfrac{n^2}{k^2} \leq \tfrac1{200}\beta n^2$ such edges.
		\item Edges in regular pairs with density less than $\rho$.
		There are at most $\rho \binom k2 \tfrac{n^2}{k^2} \leq \tfrac1{20}\beta n^2$ such edges. 
		\item Edges inside clusters. There are at most $k\cdot \binom{n/k}2\leq \frac {n^2}{2k}<\frac 1{10}\beta n^2$.
	\end{itemize}
	In total we kept all but at most $\tfrac{31}{200} \beta n^2 < \tfrac13 \beta \binom n2$ edges, so $G'$ has at least  $\left(x +  2\beta/3 \right)\binom n2 \geq (x+\beta/2) \binom k2 \left(\frac nk\right)^2$ edges.
	Since any edge of $R$ corresponds to at most $\left( \tfrac nk \right)^2$ edges of $G'$, we get $e(R) \geq (x+\beta/2) \binom k2$ as required.
\end{proof}

The following claim and corollary connect the reduced graph of $G$ and the $\varepsilon$-graph of $G(p)$, with  respect to the same partition $\Pi$. 
\begin{claim}\label{cl:EpsPptyRobust}
	Let $\varepsilon>0$ and let $G$ be a graph on $n$ vertices, and assume that $\Pi=(V_1,V_2,\dots,V_k)$ is an $\varepsilon$-regular partition of $V(G)$ with $||V_i|-|V_j||\leq 1$, for some $k \coloneqq k(\varepsilon)$.
	Then there exists $C\coloneqq C(\varepsilon,k)$ such that  for $p\geq\frac Cn$, and for every $i,j$ where $(V_i,V_j)$ is an $\varepsilon$-regular pair with $d(V_i,V_j)\geq \rho=10\varepsilon$, we have that \whp $(V_i,V_j)$ satisfies the $\varepsilon$-property in the random graph $G(p)$.
\end{claim}

\begin{proof}
	Denote $\floor{m}\leq |V_i|\leq \ceil{m}$, where $m=\tfrac nk$.
	Let $U_i\subseteq V_i$ and $U_j\subseteq V_j$ be such that $|U_i|, |U_j| \geq \varepsilon m$.
	By $\varepsilon$-regularity we have $\left|d(V_i,V_j) - d(U_i,U_j) \right| \leq \varepsilon$. Combining it with the assumption $d(V_i, V_j) \geq \rho$, we have that
	\[e_G(U_i,U_j)\geq (\rho-\varepsilon)|U_i||U_j| = 9\varepsilon |U_i||U_j|. \]
	For two disjoint subsets $U_i,U_j$ of $V(G)$, denote by $e_p(U_i,U_j)$ the random variable counting the number of edges between these sets in $G(p)$.
	Then $e_p(U_i,U_j)$ is distributed binomially with parameters $e_G(U_i, U_j)$ and $p$.
	Hence, the probability that there exist two such sets that do not satisfy the $\varepsilon$-property in $G(p)$ is at most $\binom n{\varepsilon m}^2\Pr[e_p(U_i,U_j)=0]\leq e^{-\Omega(n)}$, for, say, $p\geq \frac {\log k}{\varepsilon^2m}$.
\end{proof}

\begin{corollary}\label{cor:edgesInSgraphRobust}
		Let $0<x<1$, $0<\beta<1-x$ and let $G$ be a graph on $n$ vertices with $e(G)\geq (x+\beta) \binom n2$ and an $\varepsilon$-regular partition $\Pi=(V_1,\dots,V_k)$ of its vertices with  $\varepsilon\leq\frac {\beta}{100}$ and $k\geq \frac {2}{\varepsilon^2}$.
		Let $R\coloneqq R(G,\Pi, \rho, \varepsilon)$ be the reduced graph as  in \Cref{def:RGraph}.
		Let $p\geq \frac Cn$ where $C$ is as in the previous claim, and let $S\coloneqq S(G(p),\Pi, \varepsilon)$ be the $\varepsilon$-graph corresponding to $G(p)$, as  in \Cref{def:SGraph}.
		Then \whp $R\subseteq S$, and therefore \whp $e(S)\geq(x+\beta/2) \binom k2$.
\end{corollary}

We can now prove \Cref{thm:mainRobustness}.

\begin{proof}[Proof of \Cref{thm:mainRobustness}]
	We can assume $0<\beta<1/4$.
	Set $\varepsilon = \frac \beta{10000}$ and $k_0 = \frac {2}{\varepsilon^2}$.
	Take $n_0,K$ as given in the Regularity Lemma (\Cref{thm:RegularityLemma}), and also set $\gamma = \tfrac{2(1-48\varepsilon)}k$.
	Let $\frac {C_1}{\log(1/\beta)}\cdot \log n \leq t \leq (1-C_2\beta )n$, where $C_1,C_2$ are the absolute constants from \Cref{cor:key}.
	Let $G$ be a graph on $n\geq n_0$ vertices with $e(G) \geq ex(n, C_t) + \beta \binom n2 \geq \left(g^\gamma(t,n) + \beta/2 \right)\binom n2$ (recall that $ex(n, C_t) \geq g^\gamma(t,n)\binom n2 - 1$).
	Then by \Cref{thm:RegularityLemma} there exists an $\varepsilon$-regular partition $\Pi=(V_1,\dots,V_k)$ of $V(G)$ such that $||V_i|-|V_j||\leq 1$, with $k_0 \leq k \leq K$.
	
	We next look at the graph $G(p)$ with the same partition and consider the $\varepsilon$-graph $S\coloneqq S(G(p),\Pi, \varepsilon)$. By \Cref{cor:edgesInSgraphRobust} we have that \whp $e(S)\geq (g^\gamma(t,n)+\beta/4)\binom k2$.
	Using \Cref{cor:key} we get that \whp $G(p)$ contains a cycle of length $t$.
\end{proof}

\begin{remark}\label{re:TightRobustness}
	As mentioned in the introduction and in the beginning of this section, \Cref{thm:mainRobustness} is tight in the sense that for many values of $t$ taking a graph $G$ with $\Theta(n^2)$ extra edges above the extremal number $ex(n, C_t)$ is in fact necessary for having \whp a copy of $C_t$ in $G(p)$ where $p = \tfrac Cn$.
	However, there are values of $t$ for which only $\omega(n)$ extra edges suffice.
\end{remark}

The following claim gives a description of the cases for which \Cref{thm:mainRobustness} is tight.

\begin{claim}\label{cl:TightEven}
	If $t$ is even, or is odd with $t\geq \tfrac n2$, then adding $\beta \binom n2$ edges to the extremal amount of edges in \Cref{thm:mainRobustness} is necessary.
\end{claim}

\begin{proof}
	Assume first that $t = o(n)$ and even, and take a graph $G = G(n,p_0)$ for some $p_0 = o(1)$.
	Note that $\mathbb E[e(G)] = \Theta(n^2 p_0) \gg ex(n, C_t)$ (recall that $ex(n, C_t) = O\left(n^{1+2/t} \right)$ in this case), and furthermore, taking $G(p)$ with $p = \tfrac Cn$ for some constant $C>0$ is equivalent to sampling a graph from $G(n, p_0 p)$.
	Having $p_0 p = o(\tfrac1n)$, we get that $G(p)$ is \whp acyclic, and in particular that taking only $o(n^2)$ more than the extremal number is not enough in this case.
	
	Assume now that $t = \Theta(n)$ is either even, or odd satisfying $t \geq \tfrac n2$.
	Let $a$ be a constant such that $t \ge an$ (even or odd). 
	It is known (and an easy exercise) that for any constant $C>0$, there exists some $\alpha \coloneqq \alpha(C) > 0$ such that \whp for any $an \leq t_0 \leq n$ the graph $G(t_0, p)$ has \whp at least $\alpha n$ isolated vertices, where $p = \tfrac Cn$.
	Now, let $an \leq t < n$, let $0<\varepsilon <\alpha$ be some constant, and take $G$ to be the graph on $n$ vertices consisting of two cliques sharing exactly one vertex, one of size $(1+\varepsilon)t$, denoted by $K^1$, and the other of size $n - (1+\varepsilon)t +1$, denoted by $K^2$.
	Now take $G(p)$ and look at a subgraph of it that is induced by the vertices of $K^1$. This subgraph is exactly $G((1+\varepsilon)t,p)$ and thus \whp $G(p)[K^1]$ contains at least $\alpha n$ isolated vertices.
	Therefore, \whp $G(p)[K^1]$ does not contain any cycle of length $(1+\varepsilon)t - \alpha n < t$ or larger, and in particular $G(p)$ does not contain any cycle of length $t$ or larger.
	On the other hand, $e(G) = \binom{(1+\varepsilon)t}2 + \binom{n - (1+\varepsilon)t + 1}2 \geq \binom{t-1}2 + \binom{n-t+2}2 + \tfrac\varepsilon4 n^2 = ex(n, C_t) + \tfrac\varepsilon4 n^2$.
	Note that here we look at a graph that can be cunstructed by taking the extremal example of Woodall (see \cite{Woodall}), move $\varepsilon n$ vertice from a smaller clique to a largest clique, and adjust all relevant edges accordingly.
\end{proof}

\subsection{Robustness for odd cycles}

In this subsection we discuss the supplemental part of \Cref{cl:TightEven}, where we prove a tight robustness result for odd cycles shorter than $\frac n2$. 

\begin{proof}[Proof of \Cref{cl:TightOdd}.]
	Let $0 < \varepsilon \le \min\left\{\tfrac \beta{1000}, \tfrac1{1105^2} \right\}$, and let $k_0 = \left\lceil\frac {2}{\varepsilon^2} \right\rceil$. 
	Let $n_0,K_0$ be as given in the Regularity Lemma (\Cref{thm:RegularityLemma}).
	Let $t \in [\tfrac{C_1}{\log(1/\beta)}\log n, \left(\tfrac12 - \beta \right)n]$ be odd.
	Let $G$ be a graph on $n \geq n_0$ vertices with $e(G) \geq ex(n, C_t) + \frac 1p \cdot f(n) = \floor{\tfrac14 n^2} + \frac 1p \cdot f(n)$, where $f(n)$ is a monotone increasing function tending to infinity with $n$, and assume that $f(n)=o(n)$.
	By \cref{thm:RegularityLemma} there exists an $\varepsilon$-regular partition $\Pi = (V_1, \ldots, V_k)$ of  $V(G)$, for some $k_0 \le k \le K_0$, such that $||V_i| - |V_j|| \leq 1$ for every $i,j \in [k]$.
	Let $\rho = 10\varepsilon$ and let $R \coloneqq R(G, \Pi, \rho, \varepsilon)$ be the reduced graph (as in \Cref{def:RGraph}).
	We separate the proof into two cases, by the number of edges in $R$.
	
	\textbf{Case 1:} Assume that $e(R)> \frac 14k^2$, then by \Cref{thm:woodall} $R$ contains a cycle of an odd length $b=\floor{\left(\tfrac tn + \beta \right)k}_{odd}$, and also a triangle.
	Let $C \coloneqq C(\varepsilon)$ be as given in \Cref{cl:EpsPptyRobust}, look at the graph $G(p)$ for $p \geq \tfrac Cn$, and consider the $\varepsilon$-graph $S\coloneqq S(G(p), \Pi, \varepsilon)$.
	Recall that, by \Cref{cor:edgesInSgraphRobust}, \whp $R\subseteq S$.
	Let $C_1$ be the absolute constant from \Cref{lem:key}.
	If we have $t \in [\tfrac{C_1}{\log(1/\beta)}\log n, \frac 2k n]$, then we look at a triangle in $S$ and by \Cref{lem:key} we get that \whp cycles of all lengths in $[\tfrac{C_1}{\log(1/\beta)}\log n, \frac 2k n]$ in $G(p)$, and in particular a cycle of length $t$.
	For larger values of $t$ we consider a cycle of length $b$ in $S$.
	Using \Cref{lem:key}, as $b(1-\delta)\frac nk>t$ (with $\delta = 48\varepsilon$, as given in \Cref{lem:key}), we get that \whp $G(p)$ contains a cycle of length $\ell$, for any $\ell \in \left[\frac 2k n, b(1-\delta)\frac nk \right]$, and in particular a cycle of length $t$.
	
	\textbf{Case 2:} Assume now that $e(R) \leq \tfrac 14 k^2$.
	By following carefully the calculation in the proof of \Cref{lem:EdgeRGraphRbstns}  we also have $e(R) \ge \left(\tfrac14 - 6\varepsilon \right)k^2$.
	In addition, we may assume that $\delta(G) \ge \tfrac n5$.
	Indeed, otherwise we iteratively remove vertices from $G$ in the following way.
	Let $G_0 = G$.
	If for $i\ge 0$ we have $\delta(G_i) < \tfrac{v(G_i)}5$ then we define $G_{i+1} = G_i - v_i$ for some $v_i \in V(G_i)$ with $d_{G_i}(v_i) < \tfrac{v(G_i)}5$.
	Let $i_0$ be minimal such that $\delta(G_{i_0}) \ge \tfrac{v(G_{i_0})}5$.
	Let $\varepsilon' = \tfrac95\beta$ and denote $n' = \ceil{(1-\varepsilon')n}$.
	If $v(G_{i_0})\ge n'$, then denote $G' = G_{i_0}$ and consider $G'$ instead of $G$, as we still have $e(G') \ge \tfrac14(n')^2 + \frac 1p f(n')$.
	Otherwise, let $i_1$ be such that $v(G_{i_1}) = n'$, and denote $G'' = G_{i_1}$.
	Note that now we have $e(G'') >\frac 14n^2-\frac {n}5\cdot \varepsilon'n \ge \tfrac14 (n')^2 + \beta \binom{n'}2$.
	By \Cref{thm:mainRobustness} there exists $C' > 0$ such that for $p \geq \tfrac{C'}{n'}$ \whp the graph $G''(p)$ contains an odd cycle of length $t$ for any $\tfrac{C_1}{\log(1/\beta)}\log n' \le t\le \tfrac12 n'$.
	Taking $C = \frac {C'}{1-\varepsilon'}$ so that $p\geq \frac Cn$, we get that, in particular, \whp the graph $G(p)$ contains an odd cycle of length $t$ for any $\tfrac{C_1}{\log(1/\beta)}\log n \le t \le \left(\tfrac12 - \beta \right)n$ (as $\log n>\log n'$ and $(\frac 12-\beta)n<\frac 12 n'$).
	Hence, from now on we assume that $\delta(G) \ge \tfrac n5$, since otherwise we can consider $G'$ instead of $G$. 
	We now further separate this case into two sub-cases, by the structure of the reduced graph $R$.
	We say that a graph $H$ on $h$ vertices is $\eta$-far from being bipartite if at least $\eta h^2$ edges must be removed from $H$ in order to make it bipartite.
	Otherwise, we say that $H$ is  $\eta$-close to being bipartite.
	Take $\eta=2\varepsilon$.
	
	\textbf{Subcase 2.1:} Assume that $R$ is $\eta$-close to being bipartite, and recall that $\eta = 2\varepsilon$.
	Let $A \subset V(G)$ be such that $[A, A^c]$ is a max-cut in $G$.
	Again, by following carefully the calculation in the proof of \Cref{lem:EdgeRGraphRbstns} we get that $e_G(A,A^c) \ge \left(\tfrac14 - 6\varepsilon -\eta\right)n^2$, and thus $|A|, |A^c| \ge \left(\tfrac12 - \sqrt{6\varepsilon+\eta} \right)n = \left(\tfrac12 - \sqrt{8\varepsilon} \right)n$. 
	Recall that $e(G) \ge \lfloor \tfrac14 n^2 \rfloor + \frac {f(n)}{p}$, so w.l.o.g.\ we have $e(A) \geq \frac {f(n)}{2p}= \omega\left( \frac 1p \right)$.
	In fact, this is the only part of the proof where we use the assumption about $G$ having at least $\omega\left(\tfrac 1p \right)$ extra edges above the Tur\'{a}n number for an odd cycle.
	To obtain $G(p)$ we first note that $G(p) \supseteq G[A](p) \cup G[A,A^c](p)$.
	Furthermore, we expose the edges of $G[A,A^c](p)$ in three stages.
	We start with the edges inside $A$, and we show that \whp $G[A](p)$ contains a matching of size $\omega(1)$.
	Indeed, let $m$ be the size of a maximal matching one can find in $G[A](p)$.
	Then we have
	\begin{align*}
	\mathbb P[\text{maximal matching in } G[A](p) \text{ is of size at most } m] \le \sum_{i=0}^m \binom{e_G(A)}{i} p^i (1-p)^{e_G(A) - 2in},
	\end{align*}
	where each summand bounds the probability of having a maximal matching of size $i$, by considering the probability of having $i$ edges in $G[A](p)$, and non of the edges that share no vertex with this set of $i$ edges (as this is a mximal matching).
	As there are at least $e_G(A) - |A|\cdot 2i \ge e_G(A) - 2in$ such edges, we get this bound.
	Now, considering, say, $m = \sqrt{f(n)}$ we get
	\[\mathbb P[\text{maximal matching in } G[A](p) \text{ is of size at most } m] = o(1). \]
	Hence, \whp $m \geq \sqrt{f(n)}$ holds.
	Let $M$ be a matching in $G[A]$ of size $\floor{\sqrt{f(n)}}$.
	We now expose the edges of $G[A,A^c]$ in three stages.
	Let $p_1$ be such that $(1 - p_1)^3 = 1- p$, i.e., $p_1 = 1 - (1 - p)^{\frac13} \ge \tfrac{c_1}{n}$ for some constant $c_1 > 0$.
	Note that $G[A, A^c](p)$ is the same as taking $G_1 \cup G_2 \cup G_3$ where $G_i = G[A,A^c](p_1)$ for each $i=1,2,3$, independently.
	Recall that $[A,A^c]$ is a max-cut, and that $\delta(G) \ge \tfrac n5$, so we have that $d(v,A^c) \ge \tfrac n{10}$ for every $v\in A$.
	Moreover, at most $4\sqrt{\varepsilon}n$ vertices in $A^c$ have less than $\left(\tfrac12 - 3\sqrt{\varepsilon} \right)n$ neighbors in $A$.
	Indeed, if there are $x$ vertices in $A^c$ with less than $\left(\tfrac12 - 3\sqrt{\varepsilon} \right)n$ neighbors in $A$, then
	\[\left(\tfrac14 - 8\varepsilon \right)n^2 \le e_G(A,A^c) = \sum_{u\in A^c}d(u,A) < x\left(\tfrac12 - 3\sqrt{\varepsilon} \right)n + \left(|A^c| - x \right)|A|. \]
	Recalling that both $|A|$ and $|A^c|$ are of size at least $\left(\tfrac12 - \sqrt{8\varepsilon} \right)n$, we get that $x \le 4\sqrt{\varepsilon}n$.
	Consider $G_1 = G[A,A^c](p_1)$, and let $uv$ be an edge in the matching $M$.
	For each of $u,v$ look at the set of their neighbors in $A^c$ with degree in $G$ at least $\left(\tfrac12 - 3\sqrt{\varepsilon} \right)n$ into $A$.
	We know that there are at least $\left(\tfrac1{10} - 4\sqrt{\varepsilon} \right)n$ such neighbors for each of $u,v$ and at least $\frac 12\left(\tfrac1{10} - 4\sqrt{\varepsilon} \right)n$ neighbors for each of $u,v$ such that these sets of neighbors are disjoint.
	Hence, the probability that, in $G_1$, each of $u,v$ has at least one neighbor in $A^c$ with degree at least $\left(\tfrac12 - 3\sqrt{\varepsilon} \right)n$ into $A$, where these neighbors of $u$ are disjoint from those of $v$, is at least $1 - (1-p_1)^{\frac 12\left(\frac1{10} - 4\sqrt{\varepsilon} \right)n}$.
	This creates a path on three edges in $G_1$, with endpoints in $A^c$ of high degree into $A$.
	
	We now find, w.h.p., a set of at least $f(n)^{1/4}$ such paths, with distinct endpoints, one by one.
	Let $M' \subseteq M$ be some proper subset of edges in the matching (might be empty), and assume that for every $e\in M'$ we have found in $G_1$ a path consisting of three edges such that $e$ is the middle edge, and the endpoints of this path in $A^c$, each has at least $\left(\tfrac12 - 3\sqrt{\varepsilon} \right)n$ neighbors into $A$. 
	Denote this set of paths by $M''$.
	Now take some edge $e' = uv \in M\setminus M'$.
	Then the probability that each of $u,v$ has, in $G_1$, at least one neighbor in $A^c$ with degree at least $\left(\tfrac12 - 3\sqrt{\varepsilon} \right)n$ into $A$, where both of these neighbors are distinct and are not contained in the vertices of $M''$, is at least $1 - (1-p_1)^{\frac 12\left(\frac1{10} - 4\sqrt{\varepsilon} \right)n-2|M'|}\ge 1 - (1-p_1)^{\frac 12\left(\frac1{10} - 4\sqrt{\varepsilon} \right)n-2|M|} \ge 1 - e^{-a_1}$ for some constant $a _1> 0$.
	In total the probability that, in $G_1$, both $u$ and $v$ have at least one such neighbor (both distinct) in $A^c$ is at least $q_1\coloneqq(1-e^{-a_1})^2$.
	If we can find such a path for an edge $e' = uv \in M\setminus M'$, then we update $M'$ to contain also $e'$, and $M''$ to contain also this path, and we repeat this with a new edge of $M\setminus M'$.
	Let $B$ be the event that we succeed only at most $m_1$ times in $G_1$ (i.e., that starting with $M' = \emptyset$, we end with $|M'| \le m_1$).
	Assuming $m_1 \leq f(n)^{1/4}$, we get
	\[\mathbb P[B] = \sum_{i=0}^{m_1} \binom{|M|}{i} q_1^i (1-q_1)^{|M| - i}  = o(1) .\]
	Hence, \whp we have $m_1\ge f(n)^{1/4}$.
	That is, there are at least $f(n)^{1/4}$ edges of $M$ where, in $G_1$, each is the middle edge of a $P_3$-copy with endpoints which have at least $\left(\tfrac12 - 3\sqrt{\varepsilon} \right)n$ neighbors in $A$, and all of the endpoints are distinct.
	Denote by $M_1^*$ this set of $P_3$-copies in $G_1$, so we have $|M_1^*| \ge (f(n))^{1/4} = \omega(1)$.
	Denote by $M_1$ a subset of such $P_3$-copies of $M_1^*$ of size  $\floor{(f(n))^{1/4}}$ (note that this is smaller than $\varepsilon n$).
	Let $U$ be  the set of remaining vertices, i.e., $U = V(G)\setminus V(M_1)$, and hence $|U| \ge (1-4\varepsilon)n$.
	We further denote $U_1 = U\cap A$ and $U_2 = U\cap A^c$.
	We continue to the second exposure.
	We look at the graph $G_2 = G[A,A^c](p_1)$ and focus on $G_2[U_1, U_2]$.
	Since $G[A, A^c]$, as a bipartite graph, is missing at most $8\varepsilon n^2$ edges,  we get that \whp the pair $(A, A^c)$ has the ($5\sqrt{\varepsilon}$)-property in $G_2$, and in particular the pair $(U_1, U_2)$ has, w.h.p., the $\varepsilon''$-property in $G_2$, for some $\varepsilon''>0$ satisfying $\varepsilon'' \left(\left(\tfrac12 - \sqrt{8\varepsilon} \right)n - 2\floor{f(n)^{1/4}}\right) \le 5\sqrt{\varepsilon}\left(\tfrac12 - \sqrt{8\varepsilon} \right)n$ (more precisely, $\varepsilon''\cdot \min(|U|_1, |U_2|) = 5\sqrt{\varepsilon}\cdot\min(|A|, |A^c|)$).
	Hence, using either \Cref{prop:LargeTreeEmbd} or \Cref{prop:TreeEmbd}, we embed a copy of $T^{(r,h)}_\ell$ in $G_2[U_1, U_2]$, where $h=\ceil{\tfrac{\log (5\sqrt{\varepsilon}(1/2-\sqrt{8\varepsilon})n)}{\log r}}$, $\ell = t - 5 - 2h$, and the value of $r$ is determined by the value of $t$ in the following way.
	Take $C_1$ to be the absolute constant from \Cref{lem:key}.
	If $t\in [\tfrac{C_1}{\log(1/\beta)}\log n, 5\sqrt{\varepsilon}(1-2\sqrt{8\varepsilon})n]$, then we set $r = \floor{\tfrac1{16\cdot5\sqrt{\varepsilon}}} - 2$ and use \Cref{prop:TreeEmbd} to embed a copy of $T^{(r,h)}_\ell$, and if $t\in [5\sqrt{\varepsilon}(1-2\sqrt{8\varepsilon})n, \tfrac12 n]$ then we set $r=2$ and use \Cref{prop:LargeTreeEmbd} to embed a copy of $T^{(r,h)}_\ell$.
	(Note that we embed a tree that helps us create a cycle of an odd length up to $\tfrac12 n$, even though we only need it to be of length up to $\left(\tfrac12 - \beta \right)n$.
	We do this to ensure that eventually we get a cycle of length $\left(\tfrac12 - \beta \right)n$ even when looking at $G'$, which have $(1 - \varepsilon')n$ vertices, instead of $G$, as mentioned at the beginning of Case 2.)
	In either case we embed a $T^{(r,h)}_\ell$-copy with both leaf-sets, denoted by $L_1, L_2$, in $U_1$.
	Recall that by the definition of the tree $T^{(r,h)}_\ell$ we further know that $|L_1|, |L_2| \ge 5\sqrt{\varepsilon}\left(\tfrac12 - \sqrt{8\varepsilon} \right)n$.
	We are left with the third and last exposure.
	Let $G_3 = G[A, A^c](p_1)$.
	Let $P = (x_0, x_1, x_2, x_3)$ be some $P_3$-copy in $M_1$.
	Recall that both endpoints of $P$, i.e., $x_0, x_3$, miss at most $4\sqrt{\varepsilon}n$ vertices in $G[A]$, so $d(x_i, L_j) \ge 5\sqrt{\varepsilon}\left(\tfrac12 - \sqrt{8\varepsilon} \right)n- 4\sqrt{\varepsilon}n \ge \sqrt{\varepsilon}n$, for $i=0,3$ and $j=1,2$.
	Similarly to a previous argument, we get that the probability that in $G_3$ both $x_0$ has a neighbor in $L_1$ and $x_3$ has a neighbor in $L_2$ is at least $(1 - e^{-a_2})^2$, for some constant $a_2 > 0$.
	Note further that, as the endpoints of the $P_3$-copies in $M_1$ are all distinct, these experiments are all independent.
	Hence, in total, we get that \whp there exists some $P_3$-copy in $M_1$ which both its endpoints have neighbors in $L_1$ and $L_2$, one in each.	
	Note that $G(p) \supseteq G[A](p) \cup G_1 \cup G_2 \cup G_3$, so in particular, we get that \whp $G(p)$ contains a cycle of length $t$.
	
	\textbf{Subcase 2.2:}
	Assume now that $R$ is $\eta$-far from being bipartite, so in particular non-bipartite.
	We will show that in this case $R$ contains an odd cycle of length at least $(\frac 12-130\varepsilon)k$, and an odd cycle of length at most $61$.
	Then, we will deduce the likely existence of the desired cycle in $G$. 
	
	The first step will be to get a subgraph of $R$ that has a large minimum degree.
	We iteratively remove vertices from $R$ with degree less than $\tfrac k{10}$.
	Similarly to the argument at the beginning of Case~2, if the process has not stopped after at most $16\varepsilon k$ steps, then we get a graph $R''$ on $k' = (1 - 16\varepsilon)k$ vertices and at least $e(R) - \tfrac{8\varepsilon}{5}k^2 \ge \tfrac14 (k')^2+1$ edges, so by \Cref{thm:woodall} it contains all  cycles of lengths in  $[3,\frac 12(k'+3)]$, and in particular all  cycles of odd lengths in  $[3,(\frac 12-8\varepsilon)k]$, so we proceed as in Case 1.
	So we may assume that this process ends after at most $16\varepsilon k$ steps, with a graph denoted $R'$, on $k'\geq (1 - 16\varepsilon)k$ vertices, with $e(R')\geq (\frac 14-8\varepsilon)k'$ and $\delta(R')\ge \tfrac {k'}{10}$.
	As $\eta>\frac {8\varepsilon}{5}$, $R'$ is still non-bipartite.
	Hence we may assume that $R$ is a non-bipartite graph on $k$ vertices with minimum degree at least $\tfrac k{10}$, or otherwise we consider $R'$ instead.
	
	We next use the following lemma, which we prove later.
	\begin{lemma}\label{lem:OddCycleS}
		Let $0 < \delta' < \tfrac15$ and let $R$ be a non-bipartite graph on $k$ vertices for $k \ge \tfrac 2{(\delta')^2}$, satisfying $e(R) \geq \left(\tfrac14 -\delta' \right)k^2$ and $\delta(R)\geq \frac k{10}$.
		Then $R$ contains an odd cycle of length at least $\left(\tfrac12 - 15\delta'\right)k$ and an odd cycle of length at most 61.
	\end{lemma}

We will show now how to use \Cref{lem:OddCycleS} to complete the proof of \Cref{cl:TightOdd}.
By \Cref{lem:OddCycleS} (for $\delta'=8\varepsilon$), we have in $R'$, and thus in $R$, an odd cycle of length $b_0\in [3,61]$, and an odd cycle of length $b_1\geq \left(\tfrac12 - 120\varepsilon\right)k'\geq \left(\tfrac12 - 130\varepsilon\right)k$.
Similarly to Case 1, let $C \coloneqq C(\varepsilon)$ be as given in \Cref{cl:EpsPptyRobust}, look at the graph $G(p)$ for $p \geq \tfrac Cn$, and consider the $\varepsilon$-graph $S\coloneqq S(G(p), \Pi, \varepsilon)$.
Recall that, by \Cref{cor:edgesInSgraphRobust}, \whp $R\subseteq S$.
\Cref{lem:key} Item 2 and the cycle of length $b_0$ in $R$ give, w.h.p., cycles of all odd lengths  in $[\tfrac{C_1}{\log(1/\beta)}\log n, 3(1-48\varepsilon)\frac nk]$ in $G(p)$, where $C_1$ is the absolute constant from \Cref{lem:key}.
By \Cref{lem:key} Item 2, the cycle of length $b_1$ in $R$ gives, w.h.p., cycles of all odd lengths in $[3(1-48\varepsilon)\frac nk, \left(\tfrac12-200\varepsilon \right)n]$ in $G(p)$.
Recall that the graph we are looking at might be, instead of $G$, the graph $G'$ which has $(1-\varepsilon')n$ vertices.
As we have $\varepsilon \le \tfrac{\beta}{1000}$ and $\varepsilon' = \tfrac59 \beta$, we get $\left(\tfrac12 - 200\varepsilon \right)(1-\varepsilon')n \ge \left(\tfrac12 - \beta \right)n$.
Hence, altogether, we get that \whp $G(p)$ contains a cycle of length $t$, for any odd $t\in [\tfrac{C_1}{\log(1/\beta)}\log n, \left(\tfrac12-\beta \right)n]$.
\end{proof}

It is left to prove \Cref{lem:OddCycleS}.
For this proof we use two results.
The first is by Erd\H{o}s and Gallai \cite{ErdosGallai}, estimating the maximal number of edges in a graph containing no cycle of at least a certain length.

\begin{theorem}[\cite{ErdosGallai}, Theorem 2.7]\label{thm:ErdosGallaiCycles}
	Let $G$ be an $n$-vertex graph with more than $\left\lfloor\frac 12 (n-1) (t-1)\right\rfloor$ edges.
	Then $G$ contains a cycle of length at least $t$. 
\end{theorem}

The second result we need is by Moon \cite{MoonDiameter} (see also \cite{ErdosDiameter}), regarding the diameter of a connected graph.
For a graph $G$, let $\text{diam}(G)=\max_{\{u,v\}}d(u,v)$, where $d(u,v)$ is the length of a shortest path between $u$ and $v$ in the graph (and is equal to $\infty$ if the graph is not connected and $u, v$ are in different connected components).

\begin{theorem}[\cite{ErdosDiameter}, Theorem 1; \cite{MoonDiameter}]\label{thm:Diam}
	Let $R$ be a connected graph on $k$ vertices with minimum degree $\delta(R) \ge 2$.
	Then
	\[\emph{diam}(R) \le \left\lceil \frac{3k}{\delta(R) + 1} \right\rceil - 1 \]
\end{theorem}

We further note the following.
\begin{claim}\label{cl:Rconnectivity}
	Let $\kappa(R)$ be the connectivity of the graph $R$ (the minimum number of vertices one must remove from $R$ in order to make it disconnected), and assume that $\kappa(R) = \kappa\leq \delta 'k$ and $e(R) \ge \left(\tfrac14 - \delta' \right)k^2$.
	Then $R$ contains cycles of all lengths between 3 and $\left(\tfrac12 -15\delta'\right)k$.
\end{claim}

\begin{proof}
	Let $\delta' > 0$.
	Let $Q\subseteq V(R)$ be such that $|Q| = \kappa$, $A \cup Q \cup B = V(R)$, and $e_R(A, B) = 0$.
	Denote $v(A) = a$ for some $a$ and assume w.l.o.g.\ that $a\ge \tfrac12(k - \kappa)$. Note that $|B|+\kappa\geq \frac k{10}$ (due to the minimum degree condition), and thus $a\leq \frac 9{10}k$.
	Then, as we have $e(R) \ge \left(\tfrac14 - \delta' \right)k^2$, $e(B)\leq \binom {k-a-\kappa}2$, and $e(Q) + e(Q, R\setminus Q)\le \binom{\kappa}{2} + \kappa(k - \kappa) \le k\kappa$, we get, for $t = \left(\tfrac12 - 15\delta' \right)k$,
	\[e(A) \geq e(R) - e(B) - e(Q) - e(Q, R\setminus Q) \ge \binom{t-1}2 + \binom{a-t+2}2+1. \]
	Hence, by \Cref{thm:woodall}, $G[A]$, and thus $G$, contain all odd cycles of lengths between 3 and $\left(\tfrac12 - 15\delta' \right)k$.
\end{proof}

\begin{proof}[Proof of \Cref{lem:OddCycleS}]	
	Let $\delta' > 0$.
	By \Cref{cl:Rconnectivity} we may assume that $\kappa(R) \geq \delta'k$ (as otherwise we are done).
	We find an odd cycle of length at most 61, and an odd cycle of length at least $\left(\tfrac12 - 3\delta'\right)k$ in $R$ in three steps.
	
	\textbf{Step I:} Find a short odd cycle in $R$.
	Let $C_0 = (v_1, \ldots, v_{2t+1}, v_1)$ be a shortest odd cycle in $R$, for some integer $t$.
	By \Cref{thm:Diam} there exists some path $\tilde{P}$ of length at most $\tfrac{3k}{\delta(R)} \le 30$ from $v_1$ to $v_{t+1}$.
	Then the union of $\tilde{P}$ with the part of $C_0$ from $v_1$ to $v_{t+1}$ of the right parity creates a cycle whose length is at most $t+31$, so by $C_0$ being of minumum length we get $2t+2 \le 61$.
	
	\textbf{Step II:} Find an odd cycle in $R$ of length in $[\frac12 \delta'k,\frac 12\delta'k+120]$.
	Let $(u, v)$ be an arbitrary edge of $C_0$, and denote $R' \coloneqq R \setminus( V(C_0) \setminus \{u, v \})$.
	We start by finding a path of length at least $\frac 12 \delta'k$ in $R'$ with endpoints $u, v$.
	Note that we have $\kappa(R') \geq \kappa(R) - |V(C_0)| \geq \kappa(R) - 61>\frac 12\delta'k$ and $\delta(R') \geq \delta(R) - |V(C_0)| \geq \tfrac k{10} - 61$.
	Let $P$ be a path of length $\frac 12\delta'k$ in $R'$ starting at $v$ and avoiding $u$, and let $w\neq v$ be the other endpoint of this path.
	This path exists as $\delta(R' - \{u\}) \geq \tfrac k{10} - 62> \frac 12\delta'k$.
	Now look at the graph $R'' \coloneqq R' - V(P) + \{w\}$, and note that we have $\kappa(R'') \geq \kappa(R') - |V(P)| > 0$  and $\delta(R'') \geq \delta(R') - |V(P)| \geq \tfrac k{10} - 61 - \frac 12\delta'k >\frac k{20}$.
	By \Cref{thm:Diam} we get that $\emph{diam}(R'') \leq 60$ and in particular there exists a path $P'$ of length at most $60$ between $w$ and $u$. Then $P\cup P'$ is a path from $u$ to $v$ in $R'$ of length at least $\frac 12 \delta'k$ and at most $\frac 12\delta'k+60$.
	Recall that $C_0$ is of an odd length, so by concatenating the path $P \cup P'$ with a path between $u$ and $v$ on $C_0$ of the right parity, we get an odd cycle of length $\ell_1\in [\frac12 \delta'k,\frac 12\delta'k+120]$ in $R$.
	Denote this cycle by $C_1$.
	
	\textbf{Step III:} Find an odd cycle of length at least $(\frac12 - 3\delta') k$.
	Let $R^* \coloneqq R - V(C_1)$, and note that $e(R^*) \geq e(R) - \binom{\ell_1}2 - \ell_1(k - \ell_1) \geq \left(\tfrac14 -\tfrac32\delta' \right)k^2$.
	By \Cref{thm:ErdosGallaiCycles} we get that $R^*$ contains a cycle of length at least $\left(\tfrac12 - 2\delta' \right)k$.
	Denote this cycle by $C_2$.
	If $C_2$ is of an odd length then we are done.
	Assume that $C_2$ is of an even length, and recall that $\kappa(R) \geq \delta'k$.
	Using Menger's Theorem, we find $\ell_1$ pairwise vertex-disjoint paths from $C_1$ to $C_2$.
	There exist two such paths for which their endpoints in $C_2$ are of distance at most $\tfrac{|V(C_2)|}{|V(C_1)|} \leq \frac 2{\delta'}$ along $C_2$.
	Denote these two paths by $P_1$ and $P_2$, their endpoints in $C_2$ by $u_1, u_2$, respectively, and their endpoints in $C_1$ by $v_1, v_2$, respectively.
	Further denote by $P_3$ the long path between $u_1$ and $u_2$ on $C_2$, which has length at least $(\frac 12-2\delta')k-\frac 2{\delta'}\geq (\frac 12-3\delta')k$.
	Recall that $C_1$ is an odd cycle of length $\ell_1$, so by concatenating $P_1 \cup P_3 \cup P_2$ with a path between $v_1$ and $v_2$ on $C_1$ of the right parity, we get an odd cycle of length at least $\left(\tfrac12-3\delta' \right)k$.
\end{proof}

\section{Further results}\label{sec:extra}
\subsection{Expander graphs}\label{sec:applications}

As mentioned in the introduction, \Cref{thm:main} is also applicable to pseudo-random graphs.

For proving \Cref{cor:ndlambda}, we
use the Expander Mixing Lemma due to Alon and Chung \cite{ExpanderMixingLemma} cited below. 

\begin{theorem}[Expander Mixing Lemma \cite{ExpanderMixingLemma}]\label{thm:MixingLem}
	Let $G$ be a $d$-regular graph on $n$ vertices where $\lambda \le d$ is the second largest eigenvalue of its adjacency matrix, in absolute value.
	Then for any two disjoint subsets of vertices $A, B \subseteq V(G)$ we have
	\[ \left|e(A,B) - \tfrac dn |A||B| \right| \leq \lambda \sqrt{|A||B|}. \] 
\end{theorem}

We now verify that for suitable values of $d$ and $\lambda$, an $(n,d,\lambda)$-graph is also upper-uniform.
Concretely, an $(n, d, \lambda)$-graph is $(p, \eta)$-upper-uniform with $p = \tfrac dn$ and $\eta \geq \tfrac \lambda d$, proving~\Cref{cor:ndlambda}.

\begin{proof}[Proof of \Cref{cor:ndlambda}]
	Take $\eta, n_0, \gamma$ as given by \Cref{thm:main}.
	Let $G$ be an $(n, d, \lambda)$-graph with $n \geq n_0$ and $\tfrac d\lambda \geq \tfrac1\eta$.
	By the Expander Mixing Lemma (\Cref{thm:MixingLem}), any two subsets $A, B \subseteq V(G)$ satisfy
	\[e(A,B) \leq \tfrac dn |A| |B| + \lambda \sqrt{|A||B|}. \]
	Recalling that $\tfrac d\lambda \geq \tfrac 1\eta$ we get
	\[ e(A,B) \leq \tfrac dn |A| |B| (1+\eta) \]
	for any two subsets $A, B \subseteq V(G)$ satisfying $|A|, |B| \geq \eta n$.
	It follows that $G$ is $(\tfrac dn, \eta)$-upper-uniform.
	Note, in addition, that an $(n,d,\lambda)$-graph $G$ also satisfies $e(G) = \tfrac12 dn \geq (1 - \beta/2)\tfrac dn \binom n2$.
	Hence, the statement follows by \Cref{thm:main}.
\end{proof}

\subsection{Ramsey-type properties of sparse random graphs}\label{sec:Ramsey}

The purpose of this section is to present some immediate applications of the Key Lemma for other extremal problems in random graphs.
Given a graph $G$ with some partition of its vertices $\Pi$, Lemma~\ref{lem:key} allows us to  convert a cycle in the corresponding $\varepsilon$-graph (or the reduced graph) to a cycle in $G$ of a prescribed length.
Here, we will show how to find such a monochromatic cycle in a multicolored $\varepsilon$-graph, and thus using the Key Lemma (Lemma \ref{lem:key}) to convert it to a monochromatic cycle in $G$ of a prescribed length.
The method is very similar to the one we used to prove the main theorem (\Cref{thm:main}).
For any coloring of $G$, we use the colorful version of the Sparse Regularity Lemma to obtain a partition $\Pi$ which is $\varepsilon$-regular in \textit{every} color.
We then look at the corresponding colored $\varepsilon$-graph and show that it has almost all edges present.
Then, we need to use a deterministic Ramsey-type argument to guarantee a long monochromatic cycle in an almost complete colored graph (the $\varepsilon$-graph).
In some cases, the Ramsey-type argument already exists and we just use it as a ``black box" for our purposes, and in other cases we prove a relevant Ramsey-type argument.
This type of arguments is not always trivial, in fact, in some cases it is not even known for $r$-colored complete graphs (for example, $r$-Ramsey numbers for even cycles).
Having this monochromatic cycle in the $\varepsilon$-graph allows us to complete the argument using the Key Lemma.  

For graphs $G$ and $H$, we write $G \to_r H$ if for every $r$-coloring of the edges of $G$, there is a monochromatic copy of $H$ ($G$ is $r$-Ramsey for $H$). The \textit{$r$-Ramsey number of $H$}, denoted by $R_r(H)$, is the minimum $n$ such that $K_n\to_r H$. 
Instead of asking for a monochromatic copy of a single graph $H$, we can
consider the Ramsey property with respect to some family of graphs $\mathcal H$ (that is, Ramsey-universality).
For a family of graphs $\mathcal H$, we write $G \to_r \mathcal H$ if for every $r$-coloring of the edges of $G$, there exists a color $i$, such that there is a monochromatic copy of every $H\in \mathcal H$ of this color ($G$ is \textit{$r$-Ramsey-Universal} for $\mathcal H$).

Below we list several typical properties of random graphs, where $G$ is the random graph $G(n,p)$ for $p=\Omega\left(\frac 1n\right)$, and $\mathcal H$ is a family of long cycles.
In some cases we give short proofs (using the Key Lemma) of previously known results (see \cite{SizeRamseyHKL,SizeRamseyAlexey}).
As we will show, the maximum length of a monochromatic cycle in $G(n,p)$ is \whp asymptotically equal to the one in a colored (almost) complete graph.
This gives another example of the \textit{transference principle} discussed in the introduction.
We would like to add that, as in the case of Tur\'an numbers, all of these results can also be proved for  upper-uniform graphs and for $(n,d,\lambda)$-graphs, with appropriate parameters. It is also worth mentioning that the theorems below give an immediate linear upper bound for the \textit{size-Ramsey number} of cycles. However, in this context, better upper bounds are already known (see \cite{SizeRamseyHKL,SizeRamseyAlexey}). 

\subsubsection*{Long monochromatic even cycles in 2-colorings of random graphs}
We first combine the Ramsey-type result on 2-colorings of almost complete graphs due to Letzter (Theorem~1.3 from \cite{Shoham}, see \Cref{lem:Shoham} here), with the Key Lemma, to obtain the asymptotically optimal result about random properties of random graphs with respect to even cycles.

\begin{theorem}\label{cor:RamseyRandom2}
	For every $\beta>0$ there exists $C > 0$ such that for $p\geq \frac Cn$,  \whp $G(n,p)\to_2 \mathcal C$, for $\mathcal C=\{C_t \mid A \log n \leq t \leq (\frac 23-\beta) n,\  t \text{ is even} \}$, where $A>0$ is an absolute constant.
	In particular, \whp $G(n,p)\to_2 C_{\lfloor(\frac 23-\beta) n\rfloor_{even}}$. 
\end{theorem}

This is asymptotically best possible, as there is a 2-coloring of $E(K_n)$ with no monochromatic path (and thus cycle) on more than $2n/3$ vertices.
To see this, we let $A_1$ and $A_2$ be two disjoint sets of vertices of sizes $\frac 13 n$ and $\frac 23 n$, respectively, where all the edges in between are colored blue, the edges inside $A_1$ are colored blue, and the edges inside $A_2$ are colored red (see, e.g., also \cite{Shoham}). 
The analogous theorem for paths was proved by Letzter in \cite{Shoham}, who showed that \whp $G(n,p)\to_2 P_{(\frac 23-\beta) n}$.

\subsubsection*{Long monochromatic even cycles in 3-colorings of random graphs}
For three colors and even cycles, we will use the Ramsey-type result by Figaj and \L uczak (Lemma~3 from \cite{RamseyConnectedMatching}, see \Cref{lem:RamseyConnectedMatching} here) about the existence of a large monochromatic \textit{connected matching} in a 3-coloring of almost complete graphs.
For applying this, we note that in \Cref{lem:key}, Item 1, it is enough to find a connected component with a matching of size $b/2$ (instead of a path of length $b$, to be discussed later in more details).
We then obtain the following by applying Lemma~3 from \cite{RamseyConnectedMatching}, together with the idea of connected matchings by \L uczak \cite{LuczakConnectedMat}.

\begin{theorem}\label{cor:RamseyRandom3}
	For every $\beta>0$ there exists $C > 0$ such that for $p\geq \frac Cn$,  \whp $G(n,p)\to_3 \mathcal C$, for $\mathcal C=\{C_t \mid A \log n \leq t \leq (\frac 12-\beta) n,\  t \text{ is even} \}$, where $A>0$ is an absolute constant. In particular, \whp $G(n,p)\to_3 C_{\lfloor(\frac 12-\beta) n\rfloor_{even}}$.
\end{theorem}

This is also asymptotically optimal.
Indeed, for even $n$, we construct the following graph on $2n-3$ vertices.
We take three sets $A,B,C$ of size $\frac {n-2}2$, and one set $D$ of size $\frac n2$, and color $E(A,B)$ and $E(C,D)$ with color 1, $E(A,D)$ and $E(B,C)$ with color 2, and everything else with color 3.
Note that this graph contains no monochromatic copy of $C_n$ (see, e.g., \cite{RamseyEven3}).
The analogous theorem for paths was proved by Dudek and Pra\l at in \cite{RamseyRandomPaths3col}.

\subsubsection*{Long monochromatic even cycles in $r$-colorings of random graphs}
In general, we let
\begin{align*}
\lambda^*:=\lambda^*(r)=\sup_{\varepsilon>0}\left\{ \lambda \bigm| \substack{\exists k_0 \text{ s.t. } \forall k \ge k_0 \text{ every graph } G \text{ on } k \text{ vertices}\\ \text{with at least } (1-\varepsilon)\tbinom k2 \text{ edges has } G\to_r P_{\lambda k}}  \right\},
\end{align*}
and we then obtain the following theorem.

\begin{theorem}\label{cor:RamseyRandom-r}
Let $r>1$ be an integer. For every $\beta>0$ there exists $C > 0$ such that for $p\geq \frac Cn$,  \whp $G(n,p)\to_r \mathcal C$, for $\mathcal C=\{C_t \mid A_r \log n \leq t \leq (\lambda^*-\beta) n,\  t \text{ is even} \}$, where $A_r>0$ is a constant that depends only on $r$.
In particular, \whp $G(n,p)\to_r C_{\lfloor(\lambda^*-\beta) n\rfloor_{even}}$. 
\end{theorem}

Note that by \Cref{thm:ErdosGallai} we have $\lambda^*\geq \frac 1r$.
We believe that the value of $\lambda^*$ is equal to the value of $R_r^{-1}(k)/k$ for $r\geq 2$, where by $R_r^{-1}(k)$ we denote the inverse of the Ramsey function witn respect to paths.
This would make \Cref{cor:RamseyRandom-r} asymptotically optimal.
Note that equality holds for the cases $r=2$ and $r=3$, as witnessed in \cite{Shoham} and \cite{RamseyEven3}, respectively.
The value of $R_r^{-1}(k)/k$ is still unknown and is thought to be $\frac 1{r-1}$ for $r \ge 3$ (see, e.g. \cite{Sarkozy}, see also \cite{Charlotte} for a recent improvement).
Nevertheless, any lower bound, $\lambda'$, on $\lambda^*$,  gives us \whp $G(n,p)\to_r C_{\lfloor(\lambda'-\beta) n\rfloor_{even}}$.

\subsubsection*{Long monochromatic odd cycles in 2- and 3-colorings of random graphs}

For the case of two and three colors and odd cycles, we can give an asymptotically optimal result, using  Ramsey-type arguments due to \L uczak \cite{LuczakConnectedMat} (Lemma~8 and Lemma~9 in \cite{LuczakConnectedMat}, see \Cref{lem:AlmostRamseyOdd2} and \Cref{lem:AlmostRamseyOdd3} here).

\begin{theorem}\label{cor:RamseyRandomOddr2}
	For every $\beta>0$ there exists $C > 0$ such that for $p\geq \frac Cn$,  \whp $G(n,p)\to_2 \mathcal C$, for $\mathcal C=\{C_t \mid A \log n \leq t \leq (\frac 12-\beta) n,\  t \text{ is odd} \}$, where $A>0$ is an absolute constant.
	In particular, \whp $G(n,p)\to_2 C_{\lfloor(\frac 12-\beta) n\rfloor_{odd}}$. 
\end{theorem}

\begin{theorem}\label{cor:RamseyRandomOddr3}
	For every $\beta>0$ there exists $C > 0$ such that for $p\geq \frac Cn$,  \whp $G(n,p)\to_3 \mathcal C$, for $\mathcal C=\{C_t \mid A \log n \leq t \leq (\frac 14-\beta) n,\  t \text{ is odd} \}$, where $A>0$ is an absolute constant.
	In particular, \whp $G(n,p)\to_3 C_{\lfloor(\frac 14-\beta) n\rfloor_{odd}}$. 
\end{theorem}

These statements are indeed asymptotically optimal, as shown in the following example.
For $r=2$, split the vertices into two sets as equal as possible and color the edges between the sets with blue, and the other edges with red.
This will give an odd monochromatic cycle of length at most $\frac {n+1}2$, but not longer.
For $r=3$, we buikd a graph on $2n-2$ vertices. 
We take two disjoint copies of a 2-colored graph on $n-1$ vertices with no $C_\frac {n+1}2$, and color the edges between them in the third color.
This will give a monochromatic cycle of length at most $\frac {n-1}2$.  
These examples are known to be best possible as for large enough odd $n$ it was shown that $R_2(C_n)=2n-1$ and $R_3(C_n)=4n-3$ (see \cite{BondyErdosOdd2,KohayaOdd3,LuczakConnectedMat}).

\subsubsection*{Long monochromatic odd cycles in $r$-colorings of random graphs}
For longer odd cycles, we let
\begin{align*}
	\lambda^*_o:=\lambda^*_o(r)=\sup_{\varepsilon>0}\left\{ \lambda \bigm| \substack{\exists k_0 \text{ s.t. } \forall k \ge k_0 \text{ every graph } G \text{ on } k \text{ vertices}\\ \text{and at least } (1-\varepsilon)\tbinom k2 \text{ edges has } G\to_r C_{\lfloor{\lambda k}\rfloor_{odd}}}  \right\}.
\end{align*}
Then by using \Cref{lem:key} Item 2, we get the following.

\begin{theorem}\label{cor:RamseyRandomOdd}
	Let $r>1$ be an integer.
	For every $\beta>0$ there exists $C > 0$ such that for $p\geq \frac Cn$,  \whp $G(n,p)\to_r \mathcal C$, for $\mathcal C=\{C_t \mid A_r \log n \leq t \leq (\lambda^*_o-\beta) n,\  t \text{ is odd} \}$, where $A_r>0$ is a constant that depend only on $r$.
	In particular, \whp $G(n,p)\to_r C_{\lfloor(\lambda^*_o-\beta) n\rfloor_{odd}}$.
\end{theorem}

The value of $R_r(C_n)$ for an odd $n$ was determined exactly by Jenssen and Skokan in \cite{RamseyOddCyclesSkokan}, and is equal to $2^{r-1}(n-1)+1$. Thus it is plausible to believe that $\lambda^*_o$ should be asymptotically equal to $\frac 1{2^{r-1}}$. 
Here we present a Ramsey-type argument (see \Cref{cl:RamseyAlmost-rcol}) to show that $\lambda^*_o(r)=\Omega(\frac {1}{r2^r})$, which differs from the upper bound only by a factor proportion to $r$. This gives the following.

\begin{corollary}\label{cor:RamseyRandomOddr}
	Let $r>2$ be an integer. For every $\beta>0$ there exists $C > 0$ such that for $p\geq \frac Cn$,  \whp $G(n,p)\to_r \mathcal C$, for $\mathcal C=\{C_t \mid A_r \log n \leq t \leq (\frac {1}{r2^{r+4}}-\beta) n,\  t \text{ is odd} \}$, where $A_r>0$ is a constant that depend only on $r$. In particular, \whp $G(n,p)\to_r C_{\lfloor(\frac {1}{r2^{r+4}}-\beta) n\rfloor_{odd}}$. 
\end{corollary}

\subsubsection*{Proof ideas}

For proving the above theorems, we will use a different variant of \Cref{thm:srl}, the colorful version of the Sparse Regularity Lemma (see, e.g., \cite{RegLemColor}).

\begin{theorem}[Colorful Sparse Regularity Lemma]\label{thm:csrl}
	For any given $\varepsilon>0$, and integers $r\geq 1$ and $k_0\geq 1$, there are constants $\eta = \eta(\varepsilon, k_0) > 0$ and $K_0 = K_0(\varepsilon, k_0) \geq k_0$ such that if $G_1,\dots,G_r$ are $(p,\eta)$-upper-uniform graphs on the vertex set $V$ of size $n$ for large enough $n$, with $0<p\leq 1$, there is an $(\varepsilon, p)$-regular partition of $V$ into $k$ parts, $\Pi=(V_1,\dots,V_k)$, where $k_0\leq k \leq K_0$, such that at least $(1-\varepsilon)\binom k2$ pairs $(V_i,V_j)$ with $1\leq i<j\leq k$ are $(\varepsilon,G_\ell,p)$-regular, for each $\ell\in [r]$.
\end{theorem}

The proofs of \Cref{cor:RamseyRandom-r,cor:RamseyRandomOdd} follow the same pattern, which is an easy combination of \Cref{thm:csrl} and \Cref{lem:key}.
We sketch the details of their proofs below.

\begin{proof}[Proof sketch of \Cref{cor:RamseyRandom-r} and \Cref{cor:RamseyRandomOdd}]
We present here a proof sketch for \Cref{cor:RamseyRandom-r}, where similar arguments apply for \Cref{cor:RamseyRandomOdd}, considering $\lambda^*_o$ instead of $\lambda^*$ and an odd cycle instead of a path, mutatis mutandis.
Recall that \whp $G=G(n,p)$ is a $(p,\eta)$-upper-uniform graph for any $\eta>0$, and that \whp $e(G) \geq (1-o(1))p\binom n2$.
Let $c:E(G)\to [r]$ be an arbitrary coloring of its edges.
Let $\delta:=\delta(\beta)$ be such that every $r$-coloring of every graph on $k$ vertices with at least $(1-\delta)\binom k2$ edges contains a monochromatic path of length at least $(\lambda^*- \beta/2)k$;
such $\delta$ exists due to the definition of $\lambda^*$.

Define $G_i$ to be the graph of the $i$th color, that is, $V(G_i)=V(G)$, $E(G_i)=\{e\in E(G) \mid c(e)=i \}$.
Then $G_1,\dots,G_r$ are $(p,\eta)$-upper-uniform graphs.
For an appropriate choice of parameters $\varepsilon \coloneqq \varepsilon(\delta),k_0,\eta, \rho$, where $k_0$ is chosen also considering the definition of $\lambda^*$, by \Cref{thm:csrl} we get an $(\varepsilon, p)$-regular partition of $V$ into $k$ parts, $\Pi=(V_1,\dots,V_k)$, where $k_0\leq k \leq K_0$, such that at least $(1-\varepsilon)\binom k2$ pairs $(V_i,V_j)$ with $1\leq i<j\leq k$ are $(\varepsilon,G_t,p)$-regular, for each $t\in [r]$.

We now look at the following graph obtained from $G$, which we denote by $\tilde R$.
The vertices are $[k]$, and for any two distinct $i,j \in [k]$, we have $ij \in E(\tilde R)$ if and only if the pair $(V_i, V_j)$ is $\varepsilon$-regular in $G_t$ for every $t\in [r]$, and is of $p$-density at least $\rho \coloneqq 2r\varepsilon$ in $G$.
By the choice of $\varepsilon$ we have that $e(\tilde R)\geq  (1-\delta)\binom k2$.
We color the edges of $\tilde R$ with $r$ colors as follows: an edge $ij$ is colored with color $\nu$ if $d_{G_\nu,p}(V_i,V_j)\geq \rho /r$ (if there is more than one such color, we choose one arbitrarily), and color the rest of the edges of $\tilde R$ arbitrarily.

By the definition of $\lambda^*$ and the choice of $\delta$, we find a monochromatic path $P$ of length $({\lambda^*-\beta/2})k$.
Assume that this monochromatic path is colored with the color $\nu$ and look at the graph $G_\nu$.
Recall that $\Pi$ is an $(\varepsilon,p)$-regular partition with respect to $G_\nu$, then for $R_\nu:=R(G_\nu,\varepsilon,\rho,p)$, we get that $P$ is also contained in $R_\nu$, and thus in $S_\nu=S(G_\nu,\varepsilon,\Pi)$.
By the Key Lemma (\Cref{lem:key}), we get monochromatic even cycles of all lengths from $A_r\log n$ to $(\lambda^*-\beta)n$, all in the same color.
\end{proof}

\Cref{cor:RamseyRandom2} is a direct corollary of \Cref{cor:RamseyRandom-r} and the following lemma by Letzter. 

\begin{lemma}[Theorem~1.3 in \cite{Shoham}]\label{lem:Shoham}
	Given $0\leq \varepsilon\leq 1/64$ and large enough $k$, for every graph $G$ on $k$ vertices and $(1-\varepsilon)\binom k2$ edges, $G\to_2 P_\ell$ where $\ell = \floor{2k/3-110\sqrt{\varepsilon}k}$.
\end{lemma}

For the proof of \Cref{cor:RamseyRandom3} we need to adjust \Cref{lem:key}, Item 1.
Assume that $G$ and $S$ are as in \Cref{lem:key}, and assume that $S$ contains a tree $T$ with a matching $M$ of size $b/2$ where $b$ is even and $T$ is minimal with respect to containment.
Denote the edges of $M$ by $\{v_{2i-1}, v_{2i} \}$ for $i \in [b/2]$.
We combine the argument on connected matchings (see \cite{LuczakConnectedMat}, and also \cite{RamseyRandomPaths3col,RamseyConnectedMatching}) with a slight adaptation of the proof of the Key Lemma.
Let $W$ be a minimal closed walk on $T$ in which every edge is visited exactly twice, starting at an arbitrarily chosen root $v$ of the tree (such a walk is obtained, e.g., by a standard DFS-based argument, see, e.g., \cite{PokMonocPaths}).
In particular, every edge of $M$ is visited exactly twice in $W$.
Let $V_1, \ldots V_b$ be the clusters that correspond to the vertices $v_1, \ldots, v_b$ of the matching $M$, respectively.
For each $i\in [b]$, if $v_i$ is not a leaf of $T$ nor its neighbor in $M$ a leaf of $T$, we split $V_i$ into two subclusters, $U_i$ and $U'_i$, of equal sizes, up to possibly one vertex.
For convenience, $(U_{2i-1},U_{2i})$ corresponds to the first time that the walk $W$ is going through the edge $\{v_{2i-1},v_{2i}\}$, and $(U'_{2i-1},U'_{2i})$ corresponds to the second time.
Assume that $\{v_{2i-1}, v_{2i} \}$ is an edge in $M$ where both its vertices are not leaves of $T$, for some $i\in [b/2]$.
Then we embed a copy of $T^{(r, h)}_{\ell}$ in $(U_{2i-1}, U_{2i})$ and in $(U'_{2i-1}, U'_{2i})$, for an appropriate choice of parameters $r, h, \ell$, such that the leaf sets are embedded each in a different subcluster.
Denote the corresponding leaf sets by $L_{2i-1},L_{2i},L'_{2i-1},L'_{2i}$ and recall that they are of size $\varepsilon m$.
For edges $\{v_{2i-1}, v_{2i} \}$ in $M$ that contain a leaf (say, $v_{2i}$, for some $i\in [b/2]$, is a leaf of $T$), we embed a copy of $T^{(r,h)}_{\ell}$ in $(V_{2i-1}, V_{2i})$, for an appropriate choice of parameters $r, h, \ell$, such that both leaf sets are embedded in $V_{2i-1}$ (similarly to the proof of the Key Lemma).
Denote these leaf sets by $L_{2i-1}, L'_{2i-1}$.
Note that after embedding copies of $T^{(r,h)}_{\ell}$ in $V_1, \ldots, V_b$, each cluster still contains at least $2\varepsilon m$ vertices not touched by the tree embeddings.
We are now left with connecting the leaf sets of the embedded trees, to create a cycle of a prescribed length.

As $W$ is a closed walk, we may assume it starts with a vertex of $M$ which is a leaf of $T$ (note that any leaf of $T$ is a vertex of an edge in $M$, by minimality).
Hence, we can write $W$ as follows, $(e_1, P_1, e_2, P_2, \ldots, e_{b-1}, P_{b-1}, e_b)$, where $e_i$ is an edge of $M$, oriented according to $W$, for every $i\in [b]$, and $P_i$ is an oriented path which is a piece of the walk $W$, possibly empty, separating $e_i$ and $e_{i+1}$ and containing no edge of $M$, for every $i\in [b-1]$.
After having embedded trees in the clusters corresponding to edges of $M$ we now replace each path $P_i$ in the walk by a path of the same length and going through the same sequence of clusters.
We embed these paths one by one, as follows.
Assume for that for some $i\in [b-1]$ we have $P_i = (w_0, \ldots, w_t)$ for some $t\ge 1$ (where $w_0$ is the end vertex of $e_i$ and $w_t$ is the start vertex of $e_{i+1}$), and let $W_0, \ldots, W_t$ be the clusters corresponding to the vertices of $P_i$.
We start with a leaf set in $W_0$ of the tree corresponding to the current traversal of $e_i$, and we carefully choose the next vertex of the path in each cluster of the path $P_i$, using the $\varepsilon$-property we have for any two adjacent clusterts in $T$, so that we avoid all vertices that have already been used for embedding tress or previously embedded paths.
Note that we can always embed such paths as in each cluster $W_j$, $j\in \{0, \ldots, t\}$, we had at least $2\varepsilon m$ avaliable vertices to begin with, and in previous steps of the path ebmeddings we used at most $|T| \le k$ of them.
Doing this for every piece of the walk $W$, we obtain the desired cycle.

Thus, the proof of \Cref{cor:RamseyRandom3} follows from the statement below.

\begin{lemma}[Lemma~8 from \cite{RamseyConnectedMatching}]\label{lem:RamseyConnectedMatching}
	Let $0<\varepsilon<0.00025$ and let $k$ be large enough.
	Let $G$ be a graph on $k$ vertices with at least $(1-\varepsilon^7)\binom k2$ edges.
	Then, for every 3-coloring of the edges of $G$, there is a monochromatic component containing a matching saturating at least $(1/2 - 4.5\varepsilon)k$ vertices.
\end{lemma}

\Cref{cor:RamseyRandomOddr2,cor:RamseyRandomOddr3} follow from the statements below due to \L uczak \cite{LuczakConnectedMat}.

\begin{lemma}[Lemma 8 from \cite{LuczakConnectedMat}]\label{lem:AlmostRamseyOdd2}
	Let $0<\eta<10^{-5}$ and $k\geq \eta^{-49}$.
	Furthermore, let $G$ be a graph with $k$ vertices and at least $(1-14\eta^3)\binom k2$ edges.
	Then  every 2-coloring of the edges of $G$ leads to a monochromatic odd cycle of length at least $(\frac 12 - \frac {2\eta}{5})k$.   
\end{lemma}

\begin{lemma}[Lemma 9 from \cite{LuczakConnectedMat}]\label{lem:AlmostRamseyOdd3}
	For every $0<\eta<10^{-5}$ and $k\geq \eta^{-50}$ the following holds.
	If $G$ is a graph with $k$ vertices and at least $(1-\eta^3)\binom k2$ edges, then  every 3-coloring of the edges of $G$ leads to a monochromatic odd cycle of length at least $(\frac 14 - \frac {3\eta}{20})k$.   
\end{lemma}

For odd cycles with $r\geq 4$ colors, an equivalent Ramsey-type argument is not known yet.
Thus, we need to present an argument about the existence of a long monochromatic odd cycle in an almost compete $r$-colored graph. The proof of \Cref{cor:RamseyRandomOddr} will then immediately follow.

\begin{lemma}\label{cl:RamseyAlmost-rcol}
	Let $r>2$ be an integer and let $0<\beta<1/2^{r+1}$. Let $k$ be large enough as a function of $r$, and let $G$ be a graph on $k$ vertices with at least $(1-\beta)\binom k2$ edges. Let $c:E(G)\to [r]$ be an $r$-coloring of its edges. Then there exists a monochromatic odd cycle of length at least $\frac k{r2^{r+4}}$.
\end{lemma}

\begin{proof}
	 Let $\varepsilon=\frac 1{r2^{r+2}}$. 
	 For $i\in [r]$, let $G_i\subseteq G$ be the graph obtained by the edges colored by color $i$ according to the coloring $c$.
	 We first show that there is a color $i\in [r]$ for which $G_i$ is $\varepsilon$-far from being bipartite.
	 Assume that $G_i$ is $\varepsilon$-close to being bipartite, for every $i\in [r]$.
	 Then, as $\chi(H)\leq \Pi_{i\in [r]}\chi(H_i)$ for any graph $H$ and any pratition of its edges, we obtain that $G$ is $\varepsilon r$-close to being $2^r$-colorable.
	 Let $G'$ be a subgraph of $G$ with $e(G')\geq (1-\beta)\binom k2-\varepsilon rk^2\geq (1-\beta-2\varepsilon r)\binom k2$ such that $G'$ is  $2^r$-colorable.
	 Note that since $e(G')\geq (1-\beta-2\varepsilon r)\binom k2$, then by Tur\'an's theorem there is a clique in $G'$ of size $\frac 1{\beta+2\varepsilon r}$ and thus $\chi(G')> \frac 1{\beta+2\varepsilon r}$. Since $\beta< \frac 1{2^r}-2\varepsilon r$ we get a contradiction. 
	 
	 Let $i_0$ be such that $G_{i_0}$ is $\varepsilon$-far from being bipartite. 
	
	We next find a subgraph in $G_{i_0}$ of a linear size which is $2$-connected and is $\frac \varepsilon3$-far from being bipartite.
	For this, we define an auxiliary graph $H$.
	For each vertex $v\in V(G_{i_0})$ we look at the maximal 2-connected subgraph (a \textit{block}, sometimes also called a \textit{biconnected component}) that contains $v$ (might also be a single vertex or a bridge).
	Let $B_1,\dots ,B_h$ be those blocks (note that $|B_i\cap B_j|\leq 1$) and write $|B_i|=b_i$.
	Let $U=\{u_1,\dots,u_t\}$ be the intersection vertices, that is, for every $s\in [t]$, there are $i,j\in[h]$ such that $B_i\cap B_j=u_s$. 
	Let $H$ be the following bipartite graph (also called the \textit{block-cut tree} or the \textit{block decomposition graph}, see, e.g., \cite{West} p. 155).
	$V(H)=[h]\dot{\cup} [t]$, and $E(G)= \{ij \mid u_j\in B_i \}$. By definition, $H$ is a forest and thus $k\leq \sum b_i\leq 2k$.
	Indeed, we look at a leaf in $H$, $v_1$, by the definition it corresponds to a block $B_{j_1}$.
	Remove this block from the graph $G_{i_0}$.
	Then we obtain a graph with $h-1$ blocks and $k-b_{j_1}$ vertices.
	By induction, $\sum_{i\in [h]\setminus\{j_1 \}} b_i\leq 2(k-b_{j_1})$.
	Since $v_1$ was a leaf, we obtain that $\sum_{i\in [h]} b_i\leq 2(k-b_{j_1})+2b_{j_1}=2k$.  
	Assume that each block $B_j$ is such that at most $\frac \varepsilon 3 k b_j$ edges should be removed from it in order to make it bipartite (that is, is at most $\frac{\varepsilon k}{3b_j}$-far from being bipartite).
	Then, as $H$ is bipartite, at most  $\frac \varepsilon 3 k \sum b_j\leq \frac 23\varepsilon k^2$ edges can be removed from $G_{i_0}$ to make it bipartite.
	This is a contradiction.
	Then there is $B_j\subseteq G_{i_0}$ and is $\frac{\varepsilon k}{3 b_j}$-far from being bipartite.
	In particular this means that $\binom {b_j}2> \frac 23\varepsilon k b_j$, so we obtain $B_j$ which is $2$-connected, with $b_j\geq \frac 43\varepsilon k$.
	
	We now look at $B:=B_j$, a subgraph of $G$ which has $k' \coloneqq b_j \geq \frac 43\varepsilon k$ vertices, is 2-connected, and is $\frac{\varepsilon k}{3k'}$-far from being bipartite.
	Then in particular $e(B)\geq \frac 23\varepsilon kk'$.
	Let $C_1$ be an odd cycle in $B$.
	If $v(C_1)\geq \frac 13 \varepsilon k$ then we are done.
	Otherwise, let $B'':=B\setminus V(C_1)$ be the graph obtained by removing the vertices of $C_1$ from $B$.
	Then $v(B'')=v(B)-v(C_1)\geq k'-\frac 13 \varepsilon k$ and $e(B'')\geq \frac 13 \varepsilon k k'$.
	By \Cref{thm:ErdosGallaiCycles} we have a cycle $C_2$ of length at least $\frac 23 \varepsilon k$.
	If $C_2$ is odd then we are done.
	Otherwise, by 2-connectivity, there are two vertex-disjoint paths $P_1,P_2$ from $C_1$ to $C_2$.
	Let $P_3$ be the longer path on $C_2$ that connects $P_1$ and $P_2$. Then $v(P_1\cup P_2\cup P_3)\geq \frac 13 \varepsilon k$.
	Since $C_1$ is odd, there are two paths that connect $P_1$ and $P_2$ along $C_1$, one is even and one is odd.
	We choose the one that creates an odd cycle together with $P_1\cup P_2\cup P_3$.
	This gives an odd cycle of length at least $\frac 13 \varepsilon k$.
\end{proof}

\section*{Acknowledgements}
The second author would like to thank Oliver Riordan and Eoin Long for helpful discussions.

\small{
	
}

\end{document}